%% file: main.tex
\documentclass[11pt]{scrartcl}

\usepackage[utf8]{inputenc}

\usepackage{config}
\usepackage{titling}

\usepackage[a4paper, top=3cm, bottom=2cm, left=2.5cm, right=2.5cm, includefoot]{geometry}

\usepackage[inline]{enumitem}

\setlist[itemize]{topsep=0pt,partopsep=0pt,itemsep=0pt,parsep=0pt}
\setlist[itemize,1]{label={\small\textbullet}}
\setlist[itemize,2]{label={\tiny\textbullet}}
\setlist[itemize,3]{label=$\cdot$}
\setlist[enumerate]{topsep=0pt,partopsep=0pt,itemsep=0pt,parsep=0pt}
\setlist[enumerate,1]{label=\roman*)}
\setlist[enumerate,2]{label=\alph*)}
\setlist[enumerate,3]{label=\arabic*)}

\hypersetup{
	colorlinks=true,
	linkcolor=AO!65!black,
	citecolor=AO!65!black,
	urlcolor=AppleGreen!65!black,
	bookmarksopen=true,
	bookmarksnumbered,
	bookmarksopenlevel=2,
	bookmarksdepth=3
      }

\title{Matching Theory and Barnette's Conjecture}
\predate{}
\date{}
\postdate{}

\preauthor{}
\DeclareRobustCommand{\authorthing}{
	\begin{center}
		Maximilian Gorsky -- Technische Universit\"at Berlin\\
		\href{mailto:m.gorsky@pm.me}{m.gorsky@pm.me}\\
        Raphael Steiner\thanks{Supported by an ETH Postdoctoral Fellowship.} -- ETH Z\"urich\\
        \href{raphaelmario.steiner@inf.ethz.ch}{raphaelmario.steiner@inf.ethz.ch}\\
		Sebastian Wiederrecht\thanks{Supported by the ANR project ESIGMA (ANR-17-CE23-0010).} -- LIRMM, University of Montpellier\\
		\href{mailto:sebastian.wiederrecht@gmail.com}{sebastian.wiederrecht@gmail.com}
\end{center}}
\author{\authorthing}
\postauthor{}

\begin{document}
\maketitle

\begin{abstract}
Barnette's Conjecture claims that all cubic, 3-connected, planar, bipartite graphs are Hamiltonian.
We give a translation of this conjecture into the matching-theoretic setting.
This allows us to relax the requirement of planarity to give the equivalent conjecture that all cubic, 3-connected, Pfaffian, bipartite graphs are Hamiltonian.

A graph, other than the path of length three, is a brace if it is bipartite and any two disjoint edges are part of a perfect matching.
Our perspective allows us to observe that Barnette's Conjecture can be reduced to cubic, planar braces.
We show a similar reduction to braces for cubic, 3-connected, bipartite graphs regarding four stronger versions of Hamiltonicity.
Note that in these cases we do not need planarity.

As a practical application of these results, we provide some supplements to a generation procedure for cubic, 3-connected, planar, bipartite graphs discovered by Holton et al. [\emph{Hamiltonian Cycles in Cubic 3-Connected Bipartite Planar Graphs, JCTB, 1985}].
These allow us to check whether a graph we generated is a brace.
\end{abstract}

\section{Introduction}\label{sec:intro}

\input{introduction}

\section{Preliminaries}\label{sec:prelim}

\input{preliminaries}

\section{A Note on Hamiltonian cycles in general graphs}\label{sec:note}

\input{general}



\section{Hamiltonicity in cubic, bipartite, Pfaffian graphs}\label{sec:pfaffian}

\input{pfaffian}

\section{Generating cubic, planar braces}\label{sec:generation}

\input{generation}

\section{Discussion}\label{sec:further}

\input{furtherresearch}



\bibliographystyle{alphaurl}
\bibliography{literature}

\end{document}

%% file: introduction.tex

In 1884, Tait conjectured that every cubic, 3-connected, planar graph is Hamiltonian \cite{tait1884listing}, which Tutte disproved in 1946 \cite{tutte1946hamiltonian}.
Subsequently, Tutte conjectured that the result might hold if the graph was cubic, 3-connected, and bipartite \cite{tutte1971factors}, which was disproven by Horton \cite{bondy1976graph}.
Since all known counterexamples to Tait's conjecture are non-bipartite and all known counterexamples to Tutte's conjecture are non-planar, Barnette's conjecture from 1969 seems like a sensible compromise.
Let $\Barn$ be the class of all cubic, 3-connected, planar, bipartite graphs.

\begin{conjecture}[Barnette \cite{barnette1969conjecture}]\label{con:barnette}
    All graphs in $\Barn$ are Hamiltonian.
\end{conjecture}

This conjecture has received considerable attention and is known to be true for some subclasses of $\Barn$, but an approach to the problem in general continues to be elusive.
See Goodey's \cite{goodey1975hamiltonian} paper from 1975 for the historically most significant example of a partial solution and see the work of Bagheri et al.\ \cite{bagheri2021hamiltonian} for a current partial solution which generalises Goodey's result.
A related conjecture, also due to Barnette, which asserts that all cubic, 3-connected, planar graphs in which every face has size at most six, was recently resolved by Kardo{\v{s}} \cite{kardos2020computer}.
Many parts of the proof make heavy use of the constraint on the size of the faces, which leaves us with no direct way to translate the methods used there.
Significant effort has also been devoted to finding strengthenings of \Cref{con:barnette}; that is statements which also concern $\Barn$, but demand more than a simple Hamiltonian cycle.
For examples of strengthenings see \cite{hertel2005survey} for an overview.

Using the perspective of Matching Theory, we formulate a result reducing the \Cref{con:barnette} to a subset of $\Barn$.
In a graph $G$, we call a set of mutually disjoint edges a \emph{matching} and we call it a \emph{perfect matching} if the unions of the edges in the set contain all vertices of $G$.
A connected graph $G$ with $|\V{G}| \geq 2k+2$ is called \emph{$k$-extendable} if for every matching $F \subseteq \E{G}$ with $|F| \leq k$ there exists a perfect matching $M$ with $F \subseteq M$.
If $G$ is isomorphic to $C_4$ or a bipartite, 2-extendable graph, it is called a \emph{brace}.

\begin{lemma}\label{lem:bracehamil}
    \Cref{con:barnette} holds if and only if every cubic, planar brace is Hamiltonian.
\end{lemma}

We show that this result can be derived from a combination of statements from across the literature (see \Cref{sec:pfaffian}).
It hints at an extension of the conjecture to a larger class of graphs.
Despite Tutte's Conjecture on the matter being false, there is a substantial class of cubic, 3-connected, non-planar, bipartite graphs which are Hamiltonian if and only if Barnette's conjecture holds.
We relegate the introduction of the notion of Pfaffian graphs to \Cref{sec:pfaffian}.

\begin{theorem}\label{thm:mainthmtwo}
\Cref{con:barnette} holds if and only if all cubic, 3-connected, Pfaffian, bipartite graphs are Hamiltonian.
\end{theorem}

To prove this result, we discuss how Hamiltonian cycles interact with a certain type of cut central to structural Matching Theory.
We call a connected graph $G$ \emph{matching covered} if it is 1-extendable. 
A cut $\Cut{X}$ in a matching covered graph $G$ is called \emph{tight} if for any perfect matching $M$ of $G$, we have $|\Cut{X} \cap M| = 1$.
We call a tight cut \emph{non-trivial} if each side of the cut contains at least two vertices.
If we contract one of the sides of a tight cut into a single vertex, we call the resulting graph a \emph{tight cut contraction}.
By repeatedly searching out non-trivial tight cuts and repeating this contraction operation, we arrive at a list of graphs which do not have non-trivial tight cuts.
Such graphs are called bricks if they are non-bipartite and braces otherwise.
The procedure we allude to here is called the \emph{tight cut decomposition} due to Lov{\'a}sz \cite{lovasz1987matching}.

Focusing on matching covered graphs might at first seems like a strong restriction, but the following result by K\H{o}nig implies that all $k$-regular, bipartite graphs are in fact matching covered.

\begin{theorem}[K\H{o}nig \cite{konig1916graphen}]
    For all $k \in \N$, the edge set of a $k$-regular, bipartite graph can be partitioned into $k$ perfect matchings.
\end{theorem}
\begin{corollary}\label{cor:matcovcor}
    Any $k$-regular, connected, bipartite graph is matching covered.
\end{corollary}

In particular, if we consider a Hamiltonian graph $G$ with an even number of vertices, then it turns out that all the structure belonging to its Hamiltonian cycles is contained in the subgraph induced by those edges of $G$ that are contained in a perfect matching.
We call this subgraph the \emph{cover graph} of $G$.

\begin{observation}
    Let $G$ be a Hamiltonian graph with an even number of vertices, then every Hamiltonian cycle of $G$ can be decomposed into two perfect matchings.
\end{observation}

It turns out that many important properties are preserved under taking tight cut contractions and are thus also preserved for the braces of a matching covered, bipartite graph.
We prove that this is also true for some types of Hamiltonicity in cubic, 3-connected, bipartite graphs.
This is interesting on its own, but it also contributes to the proof of \Cref{thm:mainthmtwo}.

By $P_k$ we denote a path on $k$ vertices and note that $P_k$ has length $k-1$.
We call a graph $G$ with at least $k$ vertices \emph{$P_k$-Hamiltonian} if any path of length $k-1$ in $G$ is contained in some Hamiltonian cycle of $G$.
Furthermore, we say that a graph $G$ has the $H^-$-property, respectively the $H^{+-}$-property, if for any edge in $G$ there exists a Hamiltonian cycle in $G$ which avoids said edge, respectively for any two distinct edges $e$ and $f$ there exists a Hamiltonian cycle containing $e$ and avoiding $f$.

\begin{theorem}\label{thm:mainthmone}
Let $G$ be a cubic, $3$-connected, bipartite graph.
The following hold
\begin{enumerate}
    \item If the braces of $G$ are $P_4$-Hamiltonian, then $G$ is $P_4$-Hamiltonian.
    
    \item $G$ has the $H^-$-property if and only if the braces of $G$ have the $H^-$-property.
    
    \item $G$ is $P_3$-Hamiltonian if and only if the braces of $G$ are $P_3$-Hamiltonian.
    
    \item $G$ has the $H^{+-}$-property if and only if the braces of $G$ have the $H^{+-}$-property.
\end{enumerate}
\end{theorem}


There have also been computational efforts concerning \Cref{con:barnette}.
The first of which, by Holton et al.\ \cite{holton1984hamiltonian} in 1985, also introduced an elegant generation method for the graphs in $\Barn$ and this is still the method used in recent efforts (see \cite{brinkmann2021minimality}).
To support the relevance of \Cref{lem:bracehamil} to the effort of resolving \Cref{con:barnette}, we provide a short list of amendments to the generation procedure in \Cref{sec:generation} that allow us to keep track of whether a graph we have generated is a brace or not.
For each generated graph this requires us to save some more information.
This amounts to at most half the number of vertices of the graph.
Using this additional information we can check in linear time whether the generated graph is a brace.
The current state of the computational effort is that all graphs in $\Barn$ with at most 90 vertices are Hamiltonian, announced by Brinkmann et al.\ \cite{brinkmann2021minimality}.

More or less separate from the rest of the presentation, we present the theorem which motivated this project.
It is almost trivial to prove, but hopefully motivates further inquires into the connections between Matching Theory and the study of Hamiltonicity.

\begin{theorem} \label{thm:tightcutham}
	The bricks and braces of the cover graph of a Hamiltonian graph with an even number of vertices are Hamiltonian.
\end{theorem}

In \Cref{sec:prelim}, we start by introducing many of the basic matching-theoretic concepts needed for our approach, including some more rigorous definitions for concepts mentioned in this introduction.
We then give a short proof of \Cref{thm:tightcutham} in \Cref{sec:note}.
Following this, we get into the core of the article and tackle cubic, Pfaffian, bipartite graphs in \Cref{sec:pfaffian}.
This is then followed by a discussion on the generation of graphs in $\Barn$ and how to keep track of their tight cuts in \Cref{sec:generation}.
In \Cref{sec:further}, we discuss our results and provide some open questions.

%% file: preliminaries.tex

All graphs we consider are simple.
Since we will mostly deal with bipartite graphs, we establish that by $G = (A \cup B, E(G))$ we denote a bipartite graph $G$ with the edge set $E(G)$ and a vertex set $V(G)$, such that $A$ and $B$ are the two colour classes that partition $V(G)$.
A graph $G$ is called \emph{connected} if all pairs of vertices can be connected by a path and it is called $k$-connected if for any set $S \subseteq \V{G}$ with $|S| \le k-1$ the graph $G - S$ remains connected.
The \emph{components} of a graph are its maximal connected subgraphs.
A component containing only one vertex is called \emph{trivial}.

Given a cycle or path, we will call it \emph{even} if it contains an even number of edges and \emph{odd} otherwise.
We call a graph $G$ \emph{$k$-regular} if all vertices in $G$ have degree exactly $k$.
If a graph is 3-regular, we also call it \emph{cubic}.
A subgraph $H$ of a graph $G$ is called \emph{spanning} if it contains all vertices of $G$.
A spanning cycle is commonly called a \emph{Hamiltonian cycle}\footnote{This concept is more appropriately attributed to Kirkman \cite{kirkman1856polyhedra}, but the connection to William R.\ Hamilton is too well-entrenched to be severed now.}, which leads to a graph containing such a cycle being called \emph{Hamiltonian}.

Given a subgraph $H$ of a graph $G$, we call $H$ \emph{conformal} if $G - H$ contains a perfect matching and $H$ is called \emph{$M$-conformal} if there exists an $M \in \Perf{G}$ such that $M$ contains no edges with one endpoint in $G-H$ and the other in $H$.



A \emph{cut around $X$} in a graph $G$, for a set $X \subseteq V(G)$, is a set $\CutG{G}{X} \subseteq \E{G}$, such that an edge $uv$ is in $\CutG{G}{X}$ if and only if it has exactly one endpoint in $X$.
We will drop the index $G$ if the graph is clear from the context.
Any set of edges $F \subseteq E(G)$ is called a \emph{cut}, if there exists a set of vertices $X \subseteq V(G)$ such that $F = \Cut{X}$.
The order of a cut $\Cut{X}$ is defined as $|\Cut{X}|$.
We also set $\overline{X} = \V{G} \setminus X$ and call $X$ and $\overline{X}$ the \emph{shores} of the cut $\Cut{X}$.
If $X$ or $\overline{X}$ only contain one element, we call $\Cut{X}$ \emph{trivial}.
A cut $\Cut{X}$ with the property that $|\Cut{X} \cap M| = 1$ for all perfect matchings $M$ of $G$ is called \emph{tight}.
Given a tight cut $\Cut{X}$ in a matching covered graph $G$, we denote the result of contracting $X$ into a single vertex $c \not\in \V{G}$ and deleting all resulting parallel edges and loops as $\ContractXinGtoV{X}{G}{c}$, which is called a \emph{tight cut contraction}.

\begin{lemma}\label{lem:tightcut}
If $G$ is a matching covered graph, then the following properties of $G$ are preserved for its tight cut contractions: planarity, ($1$-)connectivity, and being bipartite.
\end{lemma}
\begin{proof}
Proving that both planarity and connectivity are preserved through tight cut contractions is easy using common methods.
The proof that being bipartite is preserved however uses an argument which we will make use of frequently.

Let $G = (A \cup B, E(G))$ be a matching covered, bipartite graph, let $\Cut{X}$ be a tight cut in $G$, and let $M$ be a perfect matching of $G$, with $ab \in M \cap \Cut{X}$.
We can assume w.l.o.g.\ that $a \in X \cap A$ and note that $ab$ is the only edge of $M$ found in $\Cut{X}$, since it is tight.
Thus if $n = |X \cap A|$, only one the vertices from $X \cap A$ is matched with a vertex in $\overline{X}$ and thus $|X \cap B| = n - 1$.

Suppose there exists an edge $uv \in \Cut{X} \setminus \{ ab \}$ with $u \in (X \cap B)$ and let $M'$ be a perfect matching of $G$ such that $uv \in M'$.
It is now easy to observe that it is impossible for $M' \cap E(\InducedSubgraph{G}{X})$ to cover the $n$ vertices in $A$, as only $n-2$ vertices from $B$ can be matched with them.
Therefore all endpoints of edges from $\Cut{X}$ which lie in $X$ also lie in $A$.
Thus we can colour the contraction vertex with the colour of the vertices in $B$ if we contract $\overline{X}$ to construct $\ContractXinGtoV{\overline{X}}{G}{c}$.
An analogous argument also works for $\ContractXinGtoV{X}{G}{c}$.
\end{proof}

From the proof above, we extract the following statement.

\begin{lemma}\label{lem:tightcutcolour}
Let $G$ be a matching covered, bipartite graph and let $\Cut{X}$ be a tight cut in $G$.
All endpoints of the edges in $\Cut{X}$ which lie in $X$ belong to the same colour class.
\end{lemma}

In general, for any graph $G$ with a perfect matching, it suffices to consider the spanning subgraph of $G$ which contains only those edges which can be found in some perfect matching of $G$, at least when we are interested in the matching-theoretic properties of $G$.
We call this graph the \emph{cover graph of $G$}.
By repeatedly contracting non-trivial tight cuts in the cover graph of $G$ and applying the same procedure to the tight cut contractions produced this way, we perform a \emph{tight cut decomposition}.
This results in a collection of graphs which cannot be decomposed further, and thus have no non-trivial tight cuts.
These graphs are called \emph{braces}, if they are bipartite, and \emph{bricks}, if they are not bipartite.
The tight cut decomposition has the notable property that it always results in the same list of bricks and braces, no matter in which order the contractions are applied \cite{lovasz1987matching}.
For our purposes it will suffice to only consider braces, since we concentrate on bipartite graphs.
There is a more useful characterisation of braces using extendability, a proof for which can for example be found in \cite{lovasz2009matching}.

\begin{theorem}\label{thm:braceequiv}
A graph is a brace if and only if it is either isomorphic to $C_4$ or $2$-extendable.
\end{theorem}

We now introduce an operation which can be seen as a reversed tight cut contraction.
Let $G_1$ and $G_2$ be two matching covered graphs with specified vertices $u \in V(G_1)$ and $v \in V(G_2)$, such that they are incident to the same number of edges, $V(G_1) \cap V(G_2) = \emptyset$, and there exists a bijection $\varphi : N_{G_1}(u) \rightarrow N_{G_2}(v)$.
We call $H$ a \emph{splice of $G_1$ and $G_2$ at $u$ and $v$}, if $H$ is the result of taking $(G_1 - u) + (G_2 - v)$ and adding the edges $\{ xy \mid x \in N_{G_1}(u) \text{ and } \varphi(x) = y \}$.
The cut $\CutG{H}{V(G_1) \setminus \{ u \}}$ is called a \emph{splicing cut}.

Since $G_1$ and $G_2$ are matching covered and the splicing cut $F$ corresponds to two trivial tight cuts, one in $G_1$ and the other in $G_2$, it is easy to see that $H$ is again matching covered and in particular for every edge $e$ in $H$ there exists a perfect matching of $H$ which contains $e$ and exactly one edge in $C$.
Additionally, the splice of two bipartite graphs is again bipartite.
Using these observations, a result by Carvalho et al.\ \cite{de2002conjecture} (see Corollary 2.22.) implies the following.

\begin{lemma}
    Let $H$ be the splice of $G_1$ and $G_2$ and let $F$ be the corresponding splicing cut in $H$.
    If $H$ is bipartite, then $F$ is tight.
\end{lemma}

Lastly, let us note two useful properties concerning $k$-extendable graphs.

\begin{theorem}[Plummer \cite{plummer1980n}]\label{thm:exttocon}
    If $G$ is $k$-extendable for some $k > 1$, then it is also $(k-1)$-extendable and $(k+1)$-connected.
\end{theorem}

\begin{theorem}[Plummer \cite{plummer1986matching}]\label{thm:extequiv}
    A bipartite graph $G = (A \cup B, E(G))$ is $k$-extendable if and only if for all distinct $a_1, \ldots , a_k \in A$ and all distinct $b_1, \ldots , b_k \in B$ the graph $G - \{ a_1, \ldots , a_k, b_1, \ldots , b_k \}$ has a perfect matching.
\end{theorem}

%% file: general.tex
We first formalise an observation concerning the interaction of tight cuts and conformal cycles.

\begin{lemma} \label{lem:cyccross}
Let $G$ be a matching covered graph with a tight cut $\Cut{X}$ and a perfect matching $M$ of $G$. For every $M$-conformal cycle $C$ with $E(C) \cap \Cut{X} \neq \emptyset$, we have $|\E{C} \cap \Cut{X}| = 2$. 
\end{lemma}
\begin{proof}
    Clearly, since $C$ is a cycle, we have $|\E{C} \cap \Cut{X}| \geq 2$.
    We can assume towards a contradiction that $|\E{C} \cap \Cut{X}| > 2$.
    Since $C$ is even, the set $(M \setminus E(C)) \cup (E(C) \setminus M)$ is a perfect matching, which contradicts the tightness of $\Cut{X}$.
\end{proof}

Though the consequence of this lemma for Hamiltonian graphs is easy to deduce, to the best of the authors' knowledge it has gone unstated so far.
The following immediately implies \Cref{thm:tightcutham}.

\begin{lemma}\label{lem:tightcutham}
	Any tight cut contraction of a the cover graph of a Hamiltonian graph with an even number of vertices is Hamiltonian.
\end{lemma}
\begin{proof}
	Let $G$ be an even graph, $\Cut{X}$ be a tight cut in $G$, and $C$ be a Hamiltonian cycle in $G$.
	It is easy to see that $E(C)$ contains two disjoint perfect matchings $M$ and $M'$ of $G$, making $C$ an $M$-conformal cycle. 
	Using \Cref{lem:cyccross}, we thus know that $|\E{C} \cap \Cut{X}| = 2$.
	Set $G' = \ContractXinGtoV{X}{G}{c}$ and $C' = \ContractXinGtoV{X}{C}{c}$, then clearly $C'$ is a Hamiltonian cycle of $G'$.
\end{proof}

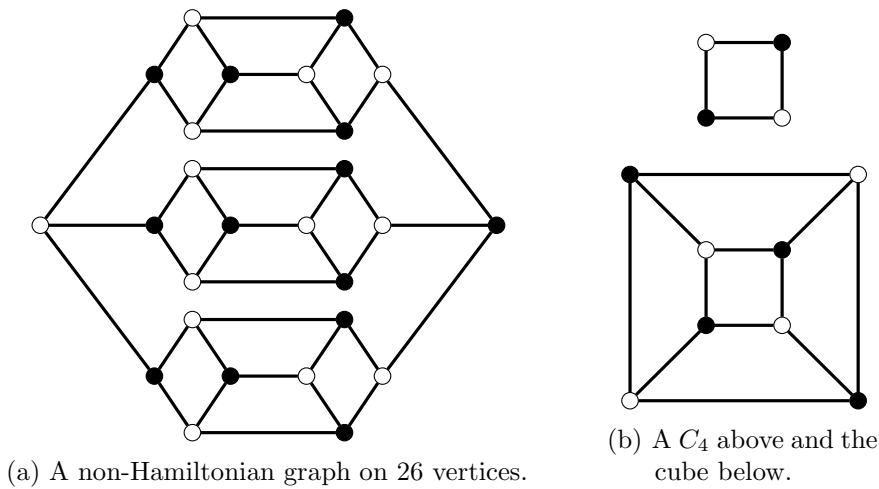
\begin{figure}[ht]
     \centering
     \begin{subfigure}[b]{0.45\textwidth}
         \centering
         \input{graphs/2conexample}
         \caption{A non-Hamiltonian graph on 26 vertices.}
         \label{fig:2conexample}
     \end{subfigure}
     \qquad
     \begin{subfigure}[b]{0.22\textwidth}
         \centering
         \input{graphs/cubeAnd4cycle}
         \caption{A $C_4$ above and the cube below.}
         \label{fig:cube}
     \end{subfigure}
     \caption{An example of a non-Hamiltonian graph with Hamiltonian braces.}
\end{figure}

Of course the reverse direction is more interesting to us, since it would allow us to construct Hamiltonian graphs using splices, but it does not hold in general.
One example for this is the smallest cubic, 2-connected, planar, bipartite graph not containing a Hamiltonian cycle found by Asano et al.\ \cite{asano1982smallest}, which we depict in \Cref{fig:2conexample} together with the braces found in its tight cut decomposition in \Cref{fig:cube}.
Clearly, both braces are Hamiltonian.

%% file: graphs/2conexample.tex
			\begin{tikzpicture}
			
			\node (A) at (0,2.75) [draw, circle, scale=0.6] {};
			\node (B) at (6,2.75) [draw, circle, scale=0.6,fill] {};
			
			\foreach\i in {0, 2, 4}
			{
				\node (\i0) at (1.5,\i+0.75) [draw, circle, scale=0.6, fill] {};
			    \node (\i1) at (2,\i+1.5) [draw, circle, scale=0.6] {};
			    \node (\i2) at (2,\i) [draw, circle, scale=0.6] {};
			    \node (\i3) at (2.5,\i+0.75) [draw, circle, scale=0.6, fill] {};
		    	\node (\i4) at (3.5,\i+0.75) [draw, circle, scale=0.6] {};
			    \node (\i5) at (4,\i+1.5) [draw, circle, scale=0.6, fill] {};
		    	\node (\i6) at (4,\i) [draw, circle, scale=0.6, fill] {};
			    \node (\i7) at (4.5,\i+0.75) [draw, circle, scale=0.6] {};
		    	
		    	\path
                    (A) edge[very thick] (\i0)
                    (\i0) edge[very thick] (\i1)
                    (\i0) edge[very thick] (\i2)
                    (\i1) edge[very thick] (\i3)
                    (\i1) edge[very thick] (\i5)
                    (\i2) edge[very thick] (\i3)
                    (\i2) edge[very thick] (\i6)
                    (\i3) edge[very thick] (\i4)
                    (\i4) edge[very thick] (\i5)
                    (\i4) edge[very thick] (\i6)
                    (\i5) edge[very thick] (\i7)
                    (\i6) edge[very thick] (\i7)
                    (\i7) edge[very thick] (B)
			    ;
			}
			\end{tikzpicture}

%% file: graphs/cubeAnd4cycle.tex
			\begin{tikzpicture}
			
			\node (V1) at (0,0) [draw, circle, scale=0.6] {};
			\node (V2) at (3,0) [draw, circle, scale=0.6, fill] {};
			\node (V3) at (3,3) [draw, circle, scale=0.6] {};
			\node (V4) at (0,3) [draw, circle, scale=0.6, fill] {};
			
			\node (U1) at (1,1) [draw, circle, scale=0.6, fill] {};
			\node (U2) at (2,1) [draw, circle, scale=0.6] {};
			\node (U3) at (2,2) [draw, circle, scale=0.6, fill] {};
			\node (U4) at (1,2) [draw, circle, scale=0.6] {};
			
			\path
			(V1) edge[very thick] (V2)
			(V2) edge[very thick] (V3)
			(V3) edge[very thick] (V4)
			(V4) edge[very thick] (V1)
			(U1) edge[very thick] (U2)
			(U2) edge[very thick] (U3)
			(U3) edge[very thick] (U4)
			(U4) edge[very thick] (U1)
			(V1) edge[very thick] (U1)
			(V2) edge[very thick] (U2)
			(V3) edge[very thick] (U3)
			(V4) edge[very thick] (U4)
			;
			
			\node (W1) at (1,3.75) [draw, circle, scale=0.6, fill] {};
			\node (W2) at (2,3.75) [draw, circle, scale=0.6] {};
			\node (W3) at (2,4.75) [draw, circle, scale=0.6, fill] {};
			\node (W4) at (1,4.75) [draw, circle, scale=0.6] {};
			
			\path
			(W1) edge[very thick] (W2)
			(W2) edge[very thick] (W3)
			(W3) edge[very thick] (W4)
			(W4) edge[very thick] (W1)
			;
			
			\end{tikzpicture}

%% file: pfaffian.tex

To prove our main theorem, we will need to discuss tight cuts in cubic, bipartite graphs a bit more.
Both lemmas that follow can be found in Section 7 of \cite{mccuaig2000even}.
The first lemma can be generalised to arbitrary $k$-regular graphs via a simple application of Menger's theorem, though the forward direction only holds when $k$ is odd.

\begin{lemma}[McCuaig \cite{mccuaig2000even}]\label{lem:tightcutdisjoint}
    Let $G$ be a cubic, $3$-connected, bipartite graph.
    A non-trivial cut in $G$ contains three mutually disjoint, pairwise non-adjacent edges if and only if it is tight.
\end{lemma}

\begin{lemma}[McCuaig \cite{mccuaig2000even}]\label{lem:tightcutpreserve}
    Let $G$ be a cubic, $3$-connected, bipartite graph with a non-trivial tight cut $\Cut{X}$.
    Then both ${\ContractXinGtoV{X}{G}{c}}$ and ${\ContractXinGtoV{\overline{X}}{G}{c}}$ are cubic, $3$-connected, and bipartite.
\end{lemma}

The following strengthenings of \Cref{con:barnette} will prove central to our efforts.

\begin{theorem}[Hertel \cite{hertel2005survey} (corollary of a result by Kelmans \cite{kelmans1986konstruktsii})]\label{thm:equivkelmans}
    \Cref{con:barnette} holds if and only if all graphs in $\Barn$ have the $H^-$-property.
\end{theorem}

\begin{theorem}[Hertel \cite{hertel2005survey}]\label{thm:equivhertel}
\Cref{con:barnette} holds if and only if all graphs in $\Barn$ are $P_4$-Hamiltonian.
\end{theorem}

It should be noted that for the cubic, bipartite graphs the $H^-$-property is equivalent to $P_3$-Hamiltonicity.
We are now ready to prove the first major lemma of this section.

\begin{lemma}\label{lem:tightcutproperty}
    Let $G$ be a cubic, $3$-connected, bipartite graph, with a non-trivial tight cut $\Cut{X}$, and let $G_1 = {\ContractXinGtoV{X}{G}{c}}$ and $G_2 = {\ContractXinGtoV{\overline{X}}{G}{c}}$ be the two tight cut contractions belonging to $\Cut{X}$.
    The following hold
    \begin{enumerate}
        \item If $G_1$ and $G_2$ are $P_4$-Hamiltonian, then $G$ is $P_4$-Hamiltonian.
    
        \item $G$ has the $H^-$-property if and only if $G_1$ and $G_2$ have the $H^-$-property.
        
        \item $G$ is $P_3$-Hamiltonian if and only if $G_1$ and $G_2$ are $P_3$-Hamiltonian.
        
        \item $G$ has the $H^{+-}$-property if and only if $G_1$ and $G_2$ have the $H^{+-}$-property.
    \end{enumerate}
\end{lemma}
\begin{proof}
    Let $E_c^i$ be the set of edges in $G_i$ which are incident with $c$.
    There are three edges in $\Cut{X}$ and these are contracted into three edges in $E_c^i$ for each $i \in \{ 1, 2 \}$.
    We say that an edge $e \in E_c^i$, for $i \in \{ 1, 2 \}$, \emph{corresponds} to the edge $f \in \Cut{X}$, if $e \cap f \neq \emptyset$.
    
    Towards the first item, suppose that $G_1$ and $G_2$ are $P_4$-Hamiltonian.
    Let $P \subseteq G$ be a path of length three in $G$.
    Suppose first that $P \subseteq G_1$ and let $H_1$ be a Hamiltonian cycle of $G_1$ with $P \subseteq G_1$.
    Let $e, f \in \Cut{X}$ be the two edges in $G$ which correspond to the two edges in $E(H) \cap E_c^1$ and let $e',f' \in E_c^2$ be the two edges in $G_2$ which correspond to $e$ and $f$ respectively.
    
    Clearly, $e'$ and $f'$ together form a path of length two and since $G_2$ is $P_4$-Hamiltonian, there therefore exists a Hamiltonian cycle $H_2$ in $G_2$ which uses both $e'$ and $f'$.
    In $G$ we can now find the Hamiltonian cycle $( ( H_1 - c ) + (H_2 - c) ) + \{ e, f \}$ which uses $P$.
    Of course an analogous construction can be carried out if $P \subseteq G_2$.
    
    If $P$ is neither contained in $G_1$ nor in $G_2$ it must use at least one edge of $e \in \Cut{X}$, with $e_i$ being the edge in $G_i$ corresponding to $e$, for $i \in \{ 1, 2 \}$.
    We note that due to the edges in $\Cut{X}$ being mutually disjoint and pairwise non-adjacent (see \Cref{lem:tightcutdisjoint}), $P$ cannot use a second edge of $\Cut{X}$.
    Let us further suppose that $P$ uses one edge $f_1$ of $E(\InducedSubgraph{G}{\overline{X}})$ and one edge $f_2$ of $E(\InducedSubgraph{G}{X})$ each.
    For our construction, we also choose an edge $g \in \Cut{X} \setminus \{ e \}$ and the corresponding edges $g_1$ and $g_2$ in $G_1$ and $G_2$.
    Since both $G_1$ and $G_2$ are $P_4$-Hamiltonian, we can now choose a Hamiltonian cycle $H_i$ in $G_i$ containing the edges $e_i$ and $g_i$ for each $i \in \{ 1, 2 \}$.
    As in the previous case, this lets us construct a Hamiltonian cycle $((H_1 - c) + (H_2 - c)) + \{ g , e \} $ for $G$ which uses $P$.
    
    Finally, we need to consider the case in which two of the edges of $P$ lie in $G_1$, or $G_2$ respectively.
    W.l.o.g.\ we can assume that $|E(P) \cap E(G_1)| = 2$ and let $e \in \Cut{X} \cap E(P)$ again be the edge of the path which lies in the tight cut, with $e_i$ being the edge in $G_i$ corresponding to $e$, for $i \in \{ 1, 2 \}$.
    We choose a Hamiltonian cycle $H_1$ in $G_1$ which uses the three edges in $(E(P) \cap E(G_1)) \cup \{ e_1 \}$, which form a path.
    Let $f$ be the edge in $H_1$ which is incident with $c$ that is distinct from $e$ and let $f'$ be the corresponding edge in $G$, with $f''$ being the edge corresponding to $f'$ in $G_2$.
    This allows us to choose a Hamiltonian cycle $H_2$ in $G_2$ which uses the edges $e_2$ and $f''$.
    A Hamiltonian cycle in $G$ which contains $P$ can then be constructed using the following expression $( (H_1 - c) + (H_2 - c) ) + \{ f', e \}$.
    
    Now we can move on to considering the second point of the statement.
    For the forward direction, we assume that $G$ has the $H^-$-property.
    Let $e \in E(G_1)$ be an edge and let us first suppose that it is not incident to $c$.
    In $G$ there exists a Hamiltonian cycle $H$ that avoids $e$.
    Clearly, ${\ContractXinGtoV{X}{H}{c}}$ is a Hamiltonian cycle of $G_1$, which still avoids $e$.
    If on the other hand $e \in E(G_1)$ is incident to $c$, then let $e' \in \Cut{X}$ be the corresponding edge in $G$.
    We can now proceed exactly as in the previous case.
    Of course, these arguments also work for any edge in $E(G_2)$.
    
    We can now concern ourselves with the backward direction and assume that both $G_1$ and $G_2$ have the $H^-$-property.
    Let $e \in E(G)$ be an edge which does not lie in $\Cut{X}$.
    W.l.o.g.\ we can assume that $e \in E(G_1)$.
    Let $H_1$ be a Hamiltonian cycle in $G_1$ which avoids $e$ and let $f,g \in E(H_1)$ be the two edges incident with $c$, let $f',g' \in \Cut{X}$ be the corresponding edges in $G$, and let $f'',g'' \in E(G_2)$ be the corresponding edges to $f'$ and $g'$ in $G_2$.
    We can now choose a Hamiltonian cycle $H_2$ in $G_2$ which avoids the edge incident with $c$ which is neither $f''$ nor $g''$.
    Since $G_2$ is cubic, by \Cref{lem:tightcutpreserve}, $H_2$ must therefore use $f''$ and $g''$, which allows us to build the Hamiltonian cycle $((H_1 - c) + (H_2 - c)) + \{ f', g' \}$ for $G$.
    
    In the final case, we have $e \in E(G) \cap \Cut{X}$.
    Let $f,g \in (\Cut{X} \setminus \{ e \})$ be the two remaining edges in $\Cut{X}$ and let $e_i,f_i,g_i$ be the corresponding edges in $G_i$ for $i \in \{ 1,2 \}$.
    For both $i \in \{ 1, 2 \}$, we can now choose a Hamiltonian cycle $H_i$ in $G_i$ which avoids $e_i$ and this allows us to construct the Hamiltonian cycle $((H_1 - c) + (H_2 - c)) + \{ f, g \}$ for $G$ which avoids $e$.
    
    For the third point of the statement, we only need to observe that a cubic, $3$-connected, bipartite graph is $P_3$-Hamiltonian if and only if it has the $H^-$-property.
    This has been noted in the literature before (see \cite{hertel2005survey}).
    
    For the forward direction of the fourth point, we assume that $G$ has the $H^{+-}$-property.
    Let $e, f \in E(G_1)$ be two distinct edges.
    We will want to include $e$ and avoid $f$.
    Clearly, as in the previous points, we can simply find an appropriate Hamiltonian cycle in $G$ using $e$, or its corresponding edge, and avoiding $f$, or its corresponding edge, and then contract it to fit $G_1$.
    This of course also works for any two edges in $G_2$.
    
    For the backward direction, assume that $G_1$ and $G_2$ have the $H^{+-}$-property.
    Again, we let $e, f \in E(G)$ be two distinct edges and our goal is to include $e$ and avoid $f$ in a Hamiltonian cycle of $G$.
    If $e,f$ or their corresponding edges are both found in $G_1$ or $G_2$ respectively, we can find an appropriate Hamiltonian cycle $H$ in $G_1$ or $G_2$ respectively.
    Then we force a Hamiltonian cycle which does not use the edge incident to $c$ whose corresponding edge is not found in $H$.
    This then allows us to combine these two cycles into a Hamiltonian cycle in $G$ which uses $e$ and avoids $f$.
    
    If this is not the case, then either $e$ or $f$ are exclusively to one of the two tight cut contractions.
    W.l.o.g.\ let us say that $e$ or its corresponding edge is found in $G_1$ and $f$ or its corresponding edge is found in $G_2$.
    In this case we choose a Hamiltonian cycle $H_2$ in $G_2$ which avoids $f$, or its corresponding edge, and then choose a Hamiltonian cycle $H_1$ in $G_1$ which uses $e$, or its corresponding edge, and avoids the edge in $E_c^1$ whose corresponding edge is not used in $H_2$.
    This again allows us to combine $H_1$ and $H_2$ appropriately.
\end{proof}

This result together with \Cref{lem:tightcutpreserve} implies \Cref{thm:mainthmone}.
The $H^{+-}$-property is relevant to the contents of \Cref{sec:generation}, but will not be needed in this section.
We note that for $P_2$-Hamiltonicity the proof idea we presented fails for a statement analogous to those in \Cref{lem:tightcutproperty}, since we need to be able to dictate where the Hamiltonian cycle leaves and where it enters.
It also does not work for $P_5$-Hamiltonicity or higher, as this involves paths which can enter and leave the cut $\Cut{X}$ in $G$ without being part of a Hamiltonian cycle themselves.
In particular, some of the graphs in $\Barn$ are not $P_5$-Hamiltonian for this reason (see \Cref{fig:p5exmp}).

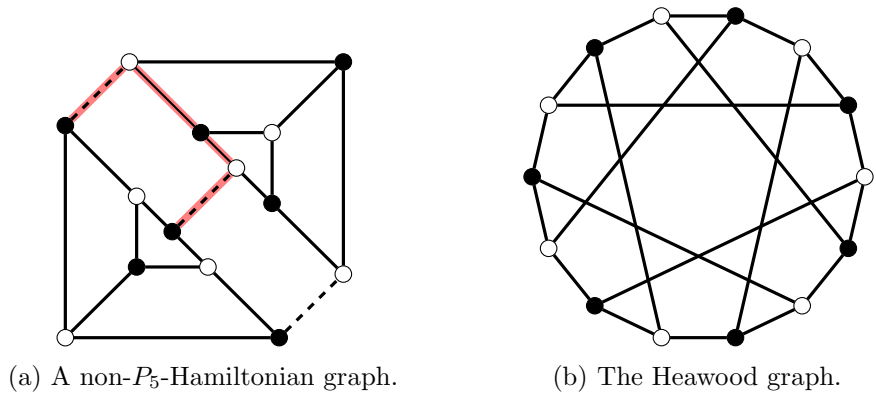
\begin{figure}[ht]
     \centering
     \begin{subfigure}[b]{0.35\textwidth}
         \centering
         \input{graphs/p5exmp}
         \caption{A non-$P_5$-Hamiltonian graph.}
         \label{fig:p5exmp}
     \end{subfigure}
     \qquad
     \begin{subfigure}[b]{0.35\textwidth}
         \centering
         \input{graphs/heawood}
         \caption{The Heawood graph.}
         \label{fig:heawood}
     \end{subfigure}
     \caption{In \Cref{fig:p5exmp}, the tight cut, which is comprised of the dashed lines, stops the highlighted $P_5$ from being in a Hamiltonian cycle.
     The graph in \Cref{fig:p5exmp} is a member of $\Barn$.}
\end{figure}

A graph $G$ is called \emph{Pfaffian} if it has an orientation of its edges such that for every perfect matching $M$ of $G$ and every $M$-conformal cycle $C$, the number of edges in $C$ oriented in the same way is odd, for either direction of traversal.
Pfaffian graphs have a myriad of connections to other problems, both in mathematics and outside of it (see \cite{mccuaig2004polya} for an excellent overview).
In particular, the following remarkable statement holds.

\begin{theorem}[Kasteleyn \cite{kasteleyn1967graph}]\label{thm:planarpfaff}
    All planar graphs are Pfaffian.
\end{theorem}

The problem of characterising Pfaffian graphs ended up being one of the central, motivating questions in Matching Theory.
As with many problems in matching theory, it can be reduced to bricks and braces.

\begin{theorem}[Vazirani and Yannakakis \cite{vazirani1989pfaffian}]\label{thm:tightcutpfaffian}
    A matching covered graph is Pfaffian if and only if its bricks and braces are Pfaffian.
\end{theorem}

Characterising Pfaffian bricks remains an open problem, but for braces McCuaig \cite{mccuaig2004polya}, and Robertson, Seymour, and Thomas \cite{robertson1999permanents} independently resolved this problem by providing the following answer.
Let $G_1,G_2,G_3$ be three bipartite graphs, such that their pairwise intersection is a cycle $C$ of length four and we have $V(G_i) \setminus V(C) \neq \emptyset$ for all $i \in \{ 1,2,3 \}$.
Further, let $S \subseteq E(C)$ be some subset of the edges of $C$.
A \emph{trisum of $G_1,G_2,G_3$ at $C$} is a graph $( \bigcup_{i=1}^3 G_i ) - S$.
For later use, we call a trisum \emph{cubic} if $S = E(C)$.
Note that the cubic trisum of three cubic, bipartite graphs is itself again cubic.

\begin{theorem}[McCuaig \cite{mccuaig2004polya}, Robertson et al.\ \cite{robertson1999permanents}]\label{thm:pfafftrisum}
A brace is Pfaffian if and only if it is the Heawood graph (see \Cref{fig:heawood}) or can be constructed from planar braces by repeated use of the trisum operation.
\end{theorem}
\begin{corollary}\label{cor:pfaffcubictrisum}
A cubic brace is Pfaffian if and only if it is the Heawood graph or can be constructed from cubic, planar braces by repeated use of the cubic trisum operation.
\end{corollary}
\begin{proof}
If a cubic brace is constructed from cubic, planar braces by repeated use of the cubic trisum operation, then it is clearly constructed from planar braces by repeated use of the trisum operation and is thus Pfaffian due to \Cref{thm:pfafftrisum}.
Of course, if we are considering the Heawood graph, there is nothing to prove.

Now, let us assume to the contrary of the forward direction that $G$ is a cubic, Pfaffian brace which cannot be constructed by repeated use of the cubic trisum operation starting from cubic planar braces.
We can assume that $G$ is non-planar, not the Heawood graph, and a minimal counterexample to the statement.

According to \Cref{thm:pfafftrisum}, there exist three Pfaffian braces $G_1$, $G_2$, and $G_3$ whose pairwise intersection is a cycle $C$ of length four, such that $G$ is the trisum of $G_1,G_2,G_3$ at $C$.
Suppose for some $i \in \{ 1,2,3 \}$ there exists some $v \in V(G_i - C)$ which has degree greater than three in $G_i$, then clearly $G$ would contain a vertex with degree greater than three and thus not be cubic.
We can thus make the narrower supposition that in fact a vertex $u \in V(C)$ has degree greater than three in $G_i$ for some $i \in \{ 1,2,3 \}$.
However, note that exactly two edges of $u$ are contained in $E(C)$ and all other edges are kept in the construction of $G$ and thus $u$ also has degree greater than three in $G$, again contradicting that it is cubic.

Thus $G_1,G_2,G_3$ are cubic, Pfaffian braces which, due to the minimality of $G$, were constructed from cubic, planar braces by repeated application of the cubic trisum operation.
Therefore the only option left is for $E(G) \cap E(C)$ to be non-empty, but this would again create vertices of degree greater than three.
This lets us conclude that $G$ was indeed constructed in the way suggested by the statement, which completes our proof.
\end{proof}

Let us now briefly observe how a Hamiltonian cycle in a cubic graph interacts with a $C_4$.

\begin{observation}\label{obs:cubicc4ham}
    Let $G$ be a cubic graph that is not isomorphic to $K_4$, let $C$ be a cycle of length four in $G$, and let $H$ be a Hamiltonian cycle in $G$.
    Then $C \cap H$ is either isomorphic to $P_4$ or to two copies of $P_2$.
\end{observation}

Using this observation, we can easily deduce the following if the graph is also $P_4$-Hamiltonian.

\begin{observation}\label{obs:c4choose}
    Let $G$ be a cubic, $P_4$-Hamiltonian graph that is not isomorphic to $K_4$ and let $C$ be a cycle of length four in $G$.
    For any path $P \subseteq C$ of length four, there exists a Hamiltonian cycle $H$ of $G$, such that $C \cap H = P$.
    Similarly, for any perfect matching $M$ of $C$, there exists a Hamiltonian cycle $H$ of $G$, such that $E(C \cap H) = M$.
\end{observation}

We can now prove a key lemma that lets us preserve $P_4$-Hamiltonicity through cubic trisums.

\begin{lemma}\label{lem:bracestosum}
    Let $G_1$, $G_2$, and $G_3$ be cubic, Pfaffian braces and let $G$ be the result of a cubic trisum of $G_1, G_2, G_3$ at $C$.
    If $G_1$, $G_2$, and $G_3$ are $P_4$-Hamiltonian, then $G$ is $P_4$-Hamiltonian.
\end{lemma}
\begin{proof}
Let $P$ be a path of length three in $G$.
Suppose $P \subset G_1$ and let $H_1$ be a Hamiltonian cycle of $G_1$ with $P \subset H_1$, which exists thanks to the $P_4$-Hamiltonicity of $G_1$.
We note that due to \Cref{obs:cubicc4ham}, the non-trivial components of $H_1 - E(C)$ must be paths of even length ending on two vertices which are adjacent on $C$.
Furthermore, we let $V(C) = \{ u_0,u_1,u_2,u_3 \}$, such that $E(C) = \{ \{ u_i, u_j \} \mid i \in \{ 0,1,2,3 \} \text{ and } i + 1 \equiv j \pmod{4} \}$.

First, we consider the case in which $H_1 - E(C)$ contains a single non-trivial component $Q$.
Note that $Q$ is a path which must have all vertices in $V(G_1) \setminus V(C)$ as its internal vertices.
W.l.o.g. we can assume that $u_0$ and $u_1$ are the endpoints of $Q$.

Now, let us make use of \Cref{obs:c4choose}.
We let $H_2$ be a Hamiltonian cycle of $G_2$ with $E(H_2) \cap E(C) = \{ u_0u_1, u_2u_3 \}$, and let $H_3$ be a Hamiltonian cycle of $G_3$ with $E(H_3) \cap E(C) = \{ u_3u_0, u_0u_1, u_1u_2 \}$.
Let $R_1$ be the component of $H_2 - \{ u_0u_3, u_1u_2 \}$ that is a path with the endpoints $u_1$ and $u_2$, and let $R_2$ be the other component of $H_2 - \{ u_0u_3, u_1u_2 \}$.
Further, let $R_3$ be the non-trivial component of $H_3 - \{ u_3u_0, u_0u_1, u_1u_2 \}$, which means it has the endpoints $u_2$ and $u_3$.
It is now easy to see that $Q + R_1 + R_3 + R_2$ is a Hamiltonian cycle of $G$.

Similarly, if $H_1 - E(C)$ contains two non-trivial components, we can choose two Hamiltonian cycles in $G_2$ and $G_3$ to complete these two paths into a Hamiltonian cycle of $G$.
We can of course proceed analogously if $P$ is entirely contained in $G_2$ or $G_3$.

Let us therefore assume that $P \subset G$, but $P \not\subset G_i$, for any $i \in \{ 1, 2, 3 \}$.
We note that, since $G$ is cubic, there must therefore exist exactly one $j \in [3]$ with $E(P) \cap E(G_j) = \emptyset$.
W.l.o.g.\ we assume that $j = 3$ and further, we can also assume that $|E(G_1) \cap E(P)| = 1$ and $|E(G_2) \cap E(P)| = 2$.
Since $P$ runs through one of the vertices in $V(C)$, where we performed the cubic 4-sum, we also know that $|P \cap V(C)| = 1$ and we can assume that $u_0 \in V(P)$.

In $G_2$, we choose a Hamiltonian cycle $H_2$ with $P \cap G_2 \subset H_2$ such that $u_0u_1 \in E(H_2)$.
Note that this immediately implies $u_2u_3 \in E(H_2)$.
Suppose now that $u_1u_2 \in E(H_2)$.
(We know that $u_0u_3 \not\in E(H_2)$, since $H_2$ enters $C$ through an edge outside of $C$ at $u_0$.)
Let $v_1$ be the neighbour of $u_1$ in $G_1$ which is not on $C$.
Inside of $G_1$, we choose a Hamiltonian cycle $H_1$ using $u_0u_1$, $u_1v_1$, and the lone edge in $E(P) \cap E(G_1)$.
Note that $H_1$ must then also use the edge $u_2u_3$, but cannot use $u_0u_3$ or $u_1u_2$, since $G_1$ is cubic.
We can now choose a Hamiltonian cycle $H_3$ in $G_3$ which uses all edges of $C$ except for $u_1u_2$.
We can now piece together a Hamiltonian cycle for $G$ which contains $P$ as in the previous cases.

If on the other hand we have $u_1u_2 \not\in E(H_2)$, we choose a Hamiltonian cycle $H_1$ in $G_1$ which uses $u_0u_3$, $u_3u_2$, and the sole edge in $E(P) \cap E(G_1)$.
For $H_1$, we observe that this implies $u_2u_1 \in E(H_1)$, since $G_1$ is cubic.
Finally, we then choose a Hamiltonian cycle $H_3$ for $G_3$ which uses all edges of $C$ except for $u_2u_3$.
As in all previous cases, this allows us to construct a Hamiltonian cycle for $G$ which uses $P$, concluding our proof.
\end{proof}

Before we prove our main theorem, we note that while the fact that it suffices to consider braces for \Cref{con:barnette} is a novel statement in the literal sense, a quick trip through the literature yields a proof of \Cref{lem:bracehamil} independent of our approach.
Kelmans proved the following in 1986.
We call a graph \emph{cyclically $k$-connected}, for $k \in \N$, if for any cut $F$ of order less than $k$, at most one component of $G - F$ contains a cycle.

\begin{theorem}[Kelmans \cite{kelmans1986konstruktsii}]\label{thm:4cyccon}
    \Cref{con:barnette} holds if and only if every cyclically $4$-connected graph in $\Barn$ is Hamiltonian.
\end{theorem}

The interested reader may find it easier to get an overview of the work of Kelmans on \Cref{con:barnette} in an article by Alt et al.\ \cite{alt2016thoughts} in which translations of his results along with sketches for their proofs are provided.
The next result was proven as part of a larger project on $n$-extendability.

\begin{theorem}[Holton and Plummer \cite{holton1988matching}]\label{thm:cyc4con}
    Let $k \in \N$ be a positive integer and let $G$ be a $(k+1)$-regular, $(k+1)$-connected, bipartite graph.
    Then if $G$ is at least cyclically $k^2$-connected, $G$ is also $k$-extendable.
\end{theorem}
\begin{corollary}\label{cor:bracecyc4con}
    If $G$ is a cubic, $3$-connected, bipartite graph, then $G$ is cyclically 4-connected if and only if $G$ is a brace.
\end{corollary}
\begin{proof}
With $k =2$, \Cref{thm:cyc4con} covers one direction.
For the other, let $G$ be a brace.
By definition, $G$ cannot have a non-trivial tight cut and thus, using \Cref{lem:tightcutdisjoint} and the $3$-connectivity of the graph, we conclude that $G$ possesses no non-trivial cuts of order 3 or less.
Hence $G$ is cyclically $4$-connected.
\end{proof}

Thus, as a corollary of \Cref{thm:cyc4con}, we know that \Cref{con:barnette} holds if and only if every brace in $\Barn$ is Hamiltonian, which was claimed in the introduction as \Cref{lem:bracehamil}.
This leads us into the proof of our main theorem.

\begin{theorem}\label{thm:mainthmpfaffian}
    The following statements are equivalent:
    \begin{enumerate}
        \item All graphs in $\Barn$ are Hamiltonian.
        
        \item All braces in $\Barn$ are Hamiltonian.
        
        \item All braces in $\Barn$ are $P_4$-Hamiltonian.
        
        \item All cubic, Pfaffian braces are $P_4$-Hamiltonian.
        
        \item All cubic, 3-connected, Pfaffian, bipartite graphs are $P_4$-Hamiltonian.
        
        \item All cubic, 3-connected, Pfaffian, bipartite graphs are Hamiltonian.
    \end{enumerate}
\end{theorem}
\begin{proof}
The equivalence of (i) and (ii) can be taken from \Cref{lem:bracehamil}.
That (i) implies (iii) is a consequence of \Cref{thm:equivhertel}.
To have (iii) imply (iv), we need to use \Cref{lem:bracestosum} and \Cref{cor:pfaffcubictrisum}.
(We note that the Heawood graph is also $P_4$-Hamiltonian.)
Since the first point of \Cref{thm:mainthmone} allows us to infer the $P_4$-Hamiltonicity of a set of cubic, bipartite, matching covered graphs from their braces, we also know that (iv) implies (v).
The implication from (v) to (vi) is trivial and (vi) to (i) is implied by \Cref{thm:planarpfaff}.
\end{proof}

We conclude this section by giving a reversed version of \Cref{lem:bracestosum}.
However, in this direction we only ask for Hamiltonicity.
This might prove useful should \Cref{con:barnette} fail, or if there exists a way to show that non-planar, cubic, Pfaffian braces are Hamiltonian.

For this purpose, we first note that there exists an earlier characterisation of Pfaffian bipartite graphs.
Since this characterisation did not yield a polynomial time algorithm for recognising Pfaffian bipartite graphs, the search continued, but for our purposes it will prove quite useful.
A \emph{bisubdivision} of a graph $H$ is the result of replacing the edges of $H$ with pairwise disjoint paths of odd length (possibly just $P_2$).

\begin{theorem}[Little \cite{little1975characterization}]\label{thm:pfaffianlittle}
    A bipartite graph with a perfect matching is Pfaffian if and only if it does not contain a bisubdivision of $K_{3,3}$ as a conformal subgraph.
\end{theorem}

We will also need a method for finding these subgraphs.
Let $C$ be a cycle of length four in a brace $G$ with $V(C) = \{ a,b,c,d \}$ and $\{a,c\},\{b,d\} \not\in E(C)$.
We call a pair of paths $L$ and $R$---with $L$ having the endpoints $a$ and $c$, and $R$ having the endpoints $b$ and $d$---a \emph{conformal cross over $C$} if $C + L + R$ is a conformal subgraph of $G$.

\begin{lemma}[Giannopoulou and Wiederrecht \cite{giannopoulou2021disjoint}]\label{lem:confcross}
    Let $G$ be a brace and let $C$ be a cycle of length four in $G$.
    There is a conformal cross over $C$ if and only if $C$ is contained in a conformal bisubdivision of $K_{3,3}$.
\end{lemma}

And with this, we turn to the proof of the lemma promised above.

\begin{lemma}\label{lem:sumtobraces}
    Let $G_1$, $G_2$, and $G_3$ be Pfaffian braces, and let $G$ be cubic and the result of a cubic trisum of $G_1, G_2, G_3$ at $C$.
    If $G$ is Hamiltonian, then $G_1$, $G_2$, and $G_3$ are Hamiltonian.
\end{lemma}
\begin{proof}
It is easy to observe that the fact that $G$ is cubic implies that $G_1$, $G_2$, and $G_3$ will also be cubic.
Now, let $H$ be a Hamiltonian cycle in $G$.
We will prove that we can construct a Hamiltonian cycle for $G_1$.
The construction for $G_2$ and $G_3$ can then be performed analogously.
Since $H$ is a Hamiltonian cycle, $H_1 = G[V(G_1)] \cap H$ consists of either one or two non-trivial components, each of which is a path with both endpoints in $V(C)$, such that $H_1 - V(C)$ is a spanning subgraph of $G_1 - V(C)$.

Suppose $P$ is the sole non-trivial component of $H_1$, where $v, w \in V(C)$ are the endpoints of $P$, such that $e = \{ v, w \} \in E(C)$, then $(P + C) - e$ is a Hamiltonian cycle in $G_1$.
Similarly, if $H_1$ consists of two paths $P_1, P_2$, each with endpoints $v_i,w_i \in V(C)$, such that $e_i = \{ v_i, w_i \} \in E(C)$ for both $i \in \{ 1, 2 \}$, then $(P_1 + P_2 + C) - \{ e_1 , e_2 \}$ is a Hamiltonian cycle in $G_1$.

On the other hand, we claim that it is impossible for $H_1$ to consist of two paths $P$ and $Q$, such that each path has endpoints which are non-adjacent in $C$.
Note that therefore $P$ and $Q$ are paths of even length, as they have endpoints of the same colour.
Let $M$ be a perfect matching of $G$ such that $M \subset E(H)$.
Due to their even length, both $P$ and $Q$ respectively have exactly one edge $e_P$, respectively $e_Q$, from $M$ which covers one of its endpoints, but is not actually an edge in $E(P)$, respectively in $E(Q)$.
In particular, the points of $e_P$ and $e_Q$ which lie on $C$ are necessarily adjacent on $C$.
Let $\{ x, y \} = E(C) \cap (e_P \cup e_Q)$ and note that $(E(H_1) \cap M) \cup \{ xy \}$ is a perfect matching of $G_1$ and in particular of $H_1 + C$.
Thus $H_1 + C$ contains a conformal cross over $C$, meaning $C$ is contained in a conformal bisubdivision of $K_{3,3}$ in $G_1$, according to \Cref{lem:confcross}.
By \Cref{thm:pfaffianlittle}, $G_1$ can therefore not be Pfaffian, contradicting our assumption.
Therefore our claim at the beginning of the paragraph is valid.

We are left with one more case in which $H_1$ has only one non-trivial component $P$, whose endpoints are not adjacent on $C$.
But in this situation, there must exists a path $Q$, whose endpoints are the two vertices in $V(C) \setminus V(P)$, which either spans $G_2 - V(P)$ or $G_3 - V(P)$.
W.l.o.g.\ we suppose that $Q \subseteq G_2 - V(P)$.
Note that in $G_3$ there exists a perfect matching $M$ for which $C$ is conformal, since $G_3$ is 2-extendable.
If we now consider that $(P + Q) - C$ is a spanning subgraph of $(G_1 + G_2) - C$, it is easy to see that $P$ and $Q$ are again paths of even length and form a conformal cross over $C$ in $G + E(C)$, with $M \setminus E(C)$ being a matching which covers the part of $G_3$ not involved in the cross.
As in the previous case, \Cref{lem:confcross} and \Cref{thm:pfaffianlittle} tell us that $G + E(C)$ is not Pfaffian.
However, according to \Cref{thm:pfafftrisum}, the graph $G + E(C)$ should in fact be Pfaffian, as it can be constructed via a trisum from three Pfaffian braces.
\end{proof}

%% file: graphs/p5exmp.tex
			\begin{tikzpicture}[scale=1.25]
			
			\node (V1) at (0,0) [draw, circle, scale=0.6] {};
			\node (V2) at (2.25,0) [draw, circle, scale=0.6, fill] {};
			\node (V4) at (0,2.25) [draw, circle, scale=0.6, fill] {};
			
			\node (U1) at (0.75,0.75) [draw, circle, scale=0.6, fill] {};
			\node (U2) at (1.5,0.75) [draw, circle, scale=0.6] {};
			\node (U3) at (1.125,1.125) [draw, circle, scale=0.6, fill] {};
			\node (U4) at (0.75,1.5) [draw, circle, scale=0.6] {};
			
			\node (W1) at (0.675,2.925) [draw, circle, scale=0.6] {};
			\node (W2) at (2.925,2.925) [draw, circle, scale=0.6, fill] {};
			\node (W4) at (2.925,0.675) [draw, circle, scale=0.6] {};
			
			\node (A1) at (2.175,2.175) [draw, circle, scale=0.6] {};
			\node (A2) at (2.175,1.425) [draw, circle, scale=0.6, fill] {};
			\node (A3) at (1.8,1.8) [draw, circle, scale=0.6] {};
			\node (A4) at (1.425,2.175) [draw, circle, scale=0.6,fill] {};
			
			\path
				(U3) edge[red, line width=3pt, opacity=0.5] (A3)
				(A3) edge[red, line width=3pt, opacity=0.5] (A4)
				(A4) edge[red, line width=3pt, opacity=0.5] (W1)
				(W1) edge[red, line width=3pt, opacity=0.5] (V4)
			;
			
			\path
			(V1) edge[very thick] (V2)
			(V4) edge[very thick] (V1)
			(U1) edge[very thick] (U2)
			(U2) edge[very thick] (U3)
			(U3) edge[very thick] (U4)
			(U4) edge[very thick] (U1)
			(V1) edge[very thick] (U1)
			(V2) edge[very thick] (U2)
			(V4) edge[very thick] (U4)
			;
			
			\path
			(W1) edge[very thick] (W2)
			(A4) edge[thick] (W1)
			(W2) edge[very thick] (W4)
			(A2) edge[very thick] (A3)
			(A3) edge[thick] (A4)
			(A4) edge[very thick] (A1)
			(A2) edge[very thick] (A1)
			(W2) edge[very thick] (A1)
			(W4) edge[very thick] (A2)
			;
			
			\path
			(V2) edge[very thick, dashed] (W4)
			(V4) edge[very thick, dashed] (W1)
			(U3) edge[very thick, dashed] (A3)
			;
			
			\end{tikzpicture}

%% file: graphs/heawood.tex
    		\begin{tikzpicture}[scale=1.25]
			
			\node (V0) at (0:0) [draw=none] {};
			
			\foreach\i in {1,3,5,7,9,11,13}
			{
				\node (V\i) at ($(V0)+({(360/14 * \i)}:1.75)$) [draw, circle, scale=0.6, fill, label={}] {};
			}
			
			\foreach\i in {2,4,6,8,10,12,14}
			{
				\node (V\i) at ($(V0)+({(360/14 * \i)}:1.75)$) [draw, circle, scale=0.6, label={}] {};
			}
			
			\foreach\i in {1,2,3,4,5,6,7,8,9,10,11,12,13}
			{
				\pgfmathtruncatemacro\iplus{\i+1}
				\path (V\i) edge[very thick] (V\iplus);
			}
			\path (V14) edge[very thick] (V1);
			
			\path
				(V1) edge[very thick] (V6)
				(V2) edge[very thick] (V11)
				(V3) edge[very thick] (V8)
				(V4) edge[very thick] (V13)
				(V5) edge[very thick] (V10)
				(V7) edge[very thick] (V12)
				(V9) edge[very thick] (V14)
			;
			
			\end{tikzpicture}

%% file: generation.tex
It is nice to know that it suffices to look at braces to resolve \Cref{con:barnette}, but that is of limited use if one does not know how to characterise these graphs in a useful way.
For this purpose, we discuss how the generation procedure for members of $\Barn$ presented by Holton, Manvel, and McKay \cite{holton1984hamiltonian} can be amended such that we can track which of the generated graphs are braces.
This removes the burden of checking whether the non-braces produced are Hamiltonian.

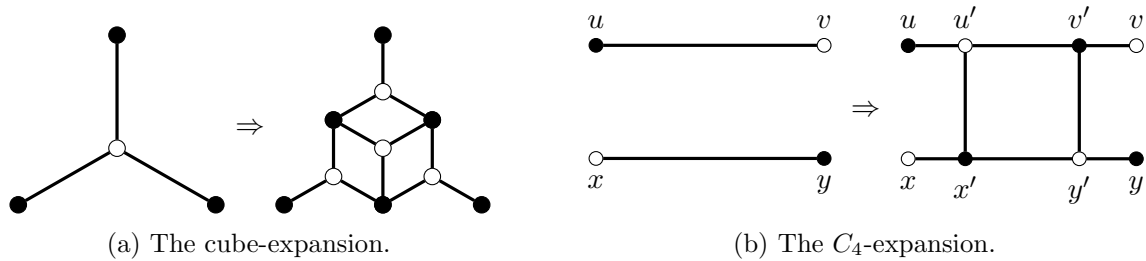
\begin{figure}[ht]
     \centering
     \begin{subfigure}[b]{0.4\textwidth}
         \centering
         \input{graphs/claw.tex} \raisebox{30pt}{$\Rightarrow$}
         \input{graphs/r4ext.tex}
         \caption{The cube-expansion.}
         \label{fig:cubeext}
     \end{subfigure}
     \qquad
     \begin{subfigure}[b]{0.5\textwidth}
         \centering
         \raisebox{3.75pt}{\input{graphs/c4expansionpre.tex}} \raisebox{37pt}{$\Rightarrow$}
         \input{graphs/c4expansion.tex}
         \caption{The $C_4$-expansion.}
         \label{fig:c4ext}
     \end{subfigure}
     \caption{The two operations needed to construct all graphs in $\Barn$ starting from the cube.
        The cube expansion can be used independent of the colours of the vertices involved.}
\end{figure}

The generation process requires two operations pictured in \Cref{fig:cubeext} and \Cref{fig:c4ext}.
We will call the operation used on a single vertex a \emph{cube-expansion} and the operation used on a pair of edges a \emph{$C_4$-expansion}.
We will need a more rigorous definition of the $C_4$-expansion than the intuition given by the illustration.

Let $G = (A \cup B, E(G)) \in \Barn$ be a matching covered, planar, bipartite graph and let $uv, xy \in E(G)$, with $u,y \in A$ and $v,x \in B$, such that $u,v,x,y$ lie on a common facial cycle $C$ and $(C - uv) - xy$ consists of two odd paths.
We call $H$ a \emph{$C_4$-expansion (in $G$) at $uv$ and $xy$} if
\begin{itemize}
    \item $\V{H} = A' \cup B'$, where $A' = A \cup \{ v', x' \}$ and $B' = B \cup \{ u', y' \}$,
    
    \item $\E{H} = (\E{G} \setminus \{ uv, xy \}) \cup \{ uu', vv', xx', yy', u'v', x'y', u'x', v'y' \}$.
\end{itemize}
We can in fact remove the requirements that $G$ is planar and along with it, that the cycle $C$ is facial, and still get a useful operation.
In this case we call $H$ a \emph{general $C_4$-expansion}.

We note that the requirements on the parity of the paths of $(C - uv) - xy$ require $u,v,x,y$ to be distinct from each other.
If we start with the cube, it turns out that cube- and $C_4$-expansions suffice to construct the entirety of $\mathcal{B}$.

\begin{theorem}[Holton et al. \cite{holton1984hamiltonian} (see Theorem 4)]
    A graph is cubic, 3-connected, planar, and bipartite if and only if it can be constructed from the cube by repeated applications of the cube- and $C_4$-expansion.
\end{theorem}

We observe that the cube-expansion creates a cut consisting of three mutually disjoint, pairwise non-adjacent edges and thus a tight cut, by \Cref{lem:tightcutdisjoint}.
One side of the tight cut created by the cube-expansion contracts into the cube, justifying its name.
It is not hard to prove that the cube-expansion preserves being cubic, $3$-connected, planar, bipartite, and Hamiltonian.
However, in the published literature concerning \Cref{con:barnette}, it seems to be the only known expansion---or rather reduction applied in reverse---which allows us to keep both the structural properties of a graph relevant to \Cref{con:barnette} and also preserve its Hamiltonicity, strengthened or otherwise.
In particular, the sticking point for all other operations tends to be that they reduce connectivity.

The fact that the cube-expansion behaves more nicely can be attributed to the fact that it is special case of a splice of two graphs in $\Barn$ which preserves planarity.
In fact, if we consider the splice of any two graphs in $\Barn$, we can easily show that this graph remains 3-connected, bipartite, and cubic.
According to \Cref{lem:tightcutdisjoint}, the splicing cut is therefore tight.
Thus the cube-expansion is simply a special case of a splice preserving planarity.

In contrast to the cube-expansion, the $C_4$-expansion is only known to preserve the structural properties of a graph in $\Barn$.
We add to the list of properties it is known to preserve by showing that it also interacts nicely with 1- and 2-extendability.
In fact, this is true for a much more general class of graphs than $\Barn$.

\begin{lemma}\label{lem:c4expansion2}
Let $H$ be a general $C_4$-expansion of a matching covered, bipartite graph $G$, then $H$ is matching covered and if $G$ is $2$-extendable, then $H$ is also $2$-extendable.
\end{lemma}
\begin{proof}
For this proof we will use the names given to the vertices involved in the $C_4$-expansion in \Cref{fig:c4ext}.
To start with, let us briefly discuss how we can extend a perfect matching $M$ of $G$ to a perfect matching of $H$.
If $uv \in M$, then we can remove $uv$ from $M$ and add $uu', vv'$, and we can proceed analogously if $xy \in M$.
Should $uv \not\in M$, we instead add $u'v'$, and again proceed analogously if $xy \not\in M$.
In the following, we will call the above steps \emph{adjusting $M$}.

Let us first prove that $H$ is matching covered and let $f \in E(H)$ be an arbitrary edge.
If $f \in E(G)$ as well, we can choose a perfect matching $M$ of $G$ containing $f$, which must exist as $G$ is matching covered.
No matter what exactly is contained in $M \cap \{ uv, xy \}$, it is easy to extend this to a perfect matching of $H$ by adjusting $M$.

Suppose instead that $f \in \{ uu', xx', vv', yy' \}$ and w.l.o.g.\ that in fact $f = uu'$.
Then we can choose a perfect matching $M$ of $G$ containing $uv$ and once more adjust $M$ to construct a perfect matching of $H$.
If instead $f \in \{ u'v', x'y' \}$, we can assume w.l.o.g.\ that $f = u'v'$.
Now we can match $u$ in $G$ to a vertex in $N_H(u) \setminus \{ v \}$ inside of a perfect matching $M$ of $G$ and by adjusting $M$, we get a perfect matching of $H$ which uses $u'v'$. 

Lastly, if $f \in \{ v'y', x'u' \}$, we can assume w.l.o.g.\ that $f = v'y'$ and take a perfect matching $M$ of $(G - u) - x$, which must exist according to \Cref{thm:extequiv}.
We have $M \subseteq E(H)$ and only $u, u', v', y', x, x'$ remain unmatched, allowing us the add the edges $uu', v'y', xx'$ to $M$ to complete a perfect matching of $M$ containing $v'y'$.

Next, let us assume that $G$ is a brace and let $F \subseteq E(H)$ be a matching containing two edges.
There are three cases to consider.

Suppose that $F \subseteq E(G)$, then there exists a perfect matching $M$ of $G$ with $F \subseteq M$ and we can again adjust $M$.
If we instead have $|F \cap E(G)| = 1$, we will choose a perfect matching of $M$ containing the edge $ab \in F \cap E(G)$ and have the freedom to fix another edge of $G$ which must be contained.
We will give a short list of the options (up to symmetry), under the assumption that $a,b \not\in \{ u,v,x,y \}$ and explain how we can change $M$ to get a perfect matching for $H$.
\begin{itemize}
    \item $uu' \in F$: Demand $ab, uv \in M$ and adjust $M$.
    
    \item $u'v' \in F$: Demand $ab, uu'' \in M$, for some $u'' \in N_H(u) \setminus \{ u' \}$, and adjust $M$.
    
    \item $u'x' \in F$: Use \Cref{thm:extequiv} to get a perfect matching matching $M'$ of $G - \{ a,b,v,y \}$ and extend this to a perfect matching $M' \cup \{ ab, u'x', vv', yy' \}$ of $H$.
\end{itemize}

Now suppose that $|\{ a,b \} \cap \{ u,v,x,y \}| = 1$.
We assume w.l.o.g. that $a = u$.
Again we list the resulting cases (up to symmetry) and how to modify $M$.
Note that $uu' \not\in F$.
\begin{itemize}
    \item $v'y' \in F$: Choose a $y'' \in N_G(y) \setminus \{ b, x \}$, which must exist since $G$ is $3$-connected by \Cref{thm:exttocon} and thus every vertex in $G$ has at least degree three.
    We can now demand that $ab, yy'' \in M$ and thus $uv, xy \not\in M$, which allows us to construct the perfect matching $M \cup \{ u'x', v'y' \} $ of $H$.
    Clearly, this construction resolves all cases in which $|F \cap \{ u'v', v'y', y'x', x'u' \}| = 1$, since the $C_4$ created by the $C_4$-expansion is conformal for the matching we constructed.
    
    \item $vv' \in F$: We can reuse the construction presented for the case $u'x' \in F$ in the previous list.
    This can also be used if $yy' \in F$.
    
    \item $xx' \in F$: Demand $ab, xy \in M$, which also means $uv \not\in M$, and adjust $M$.
\end{itemize}

Suppose that $|\{ a,b \} \cap \{ u,v,x,y \}| = 2$.
In this case, we can assume w.l.o.g.\ that $ab = ux$ and therefore know that $uu', xx' \not\in F$.
Thus to find matchings containing $ab$ and an edge of $\{ u'v', v'y', y'x', x'u' \}$, it suffices to demand that $ab \in M$ (see the first item of the previous list).
If $vv' \in F$ (or analogously $yy' \in F$), we can find a perfect matching $M$ of $G - \{ u,v,x,y \}$ using \Cref{thm:extequiv} and construct the perfect matching $M \cup \{ ux, u'x', vv', yy' \}$ of $H$.

In the third and final case, we have $F \subseteq (\E{H} \setminus \E{G})$ and this creates three subcases up to symmetry: $uu' \in F$, $u'v' \in F$, and $u'x' \in F$.
In the first, if $F \cap \{ vv', xx', yy' \} \neq \emptyset$, we can choose a perfect matching $M$ of $G$ with $uv, xy$ and adjust $M$.
If we have $x'y' \in F$, we choose a perfect matching $M$ with $uv, xx'' \in M$, where $x'' \in N_G(x) \setminus \{ x' \}$ and adjust $M$.
Should $v'y' \in F$, we can just take a perfect matching $M$ of $(G - u) - x$ and construct the perfect matching $M \cup \{ uu', v'y', xx' \}$ for $H$.
This covers all possibilities, since $uu' \in F$ implies that $u'v', u'x' \not\in F$.

For the second subcase, we know that $uu',u'x',vv',v'y' \not\in F$, since $u'v' \in F$.
If $xx' \in F$, or $yy' \in F$, we are again in the first subcase and if $x'y' \in F$, we can simply take a perfect matching of $G$ avoiding both $uv$ and $xy$ and then expand it with the edges $u'v', x'y'$.
The only part of the third subcase that does not fall into the first subcase can then be solved analogously to the second subcase.
This concludes our proof.
\end{proof}

\begin{figure}[ht]
    \centering
    \input{graphs/bigbarnetteexample}
     \caption{An example of a brace in $\Barn$.
        If the $C_4$-expansion at the red edges is reversed, we get a graph in $\Barn$ with three laminar tight cuts, marked with dashed lines.
        Should we use a cube expansion at the blue vertex, we gain a tight cut that cannot be disrupted by the $C_4$-expansion at the top of the graph.}
     \label{fig:bigexample}
\end{figure}
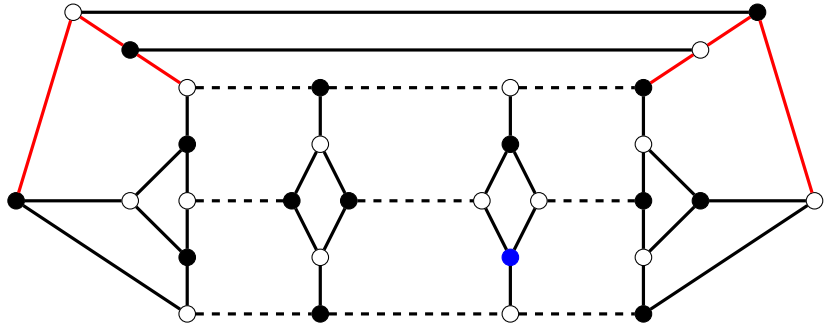

At this point, we note that it is entirely possible to use several cube-expansions on a graph in $\Barn$ and then apply a single $C_4$-expansion to build a brace.
This then begs the question of how and when a $C_4$-expansion destroys tight cuts.
Some intuition on this is given in \Cref{fig:bigexample}, where we see that a large number of tight cuts can be disrupted if a $C_4$-expansion touches both shores of those tight cuts, but the $C_4$-expansion does not destroy tight cuts for which it takes place entirely on one of the shores.

For the remainder of this section, we will work towards formalising this intuition.
We observe that the two edges used for the $C_4$-expansion cannot be part of the same tight cut.

\begin{lemma}\label{lem:c4exptight}
Let $G$ be a cubic, bipartite graph and let $H$ be a $C_4$-expansion of $G$ at $uv$ and $xy$.
No tight cut in $G$ can contain both $uv$ and $xy$.
\end{lemma}
\begin{proof}
Once more, we will use the names provided by \Cref{fig:c4ext} in this proof.

Suppose that there exists a tight cut $\Cut{X}$ with $uv, xy \in \Cut{X}$.
Note that according to \Cref{lem:tightcutcolour}, we can assume w.l.o.g.\ that $u,y \in X$.
Let $C$ be the cycle on which $u,v,x,y$ lie together.
Note that $E(C) \cap \Cut{X}$ must be even and $|E(C) \cap \Cut{X}| \geq 2$.
Thanks to \Cref{lem:tightcutdisjoint}, we also know that $|\Cut{X}| = 3$ and we can therefore conclude that $E(C) \cap \Cut{X} = \{ uv, xy \}$.
Thus, if $P$ and $Q$ are the two path components of $(C - uv) - xy$, then $P$ and $Q$ lie in distinct components of $G - \Cut{X}$.
W.l.o.g.\ we can assume that $u, y \in V(P)$ and in particular this means that $u$ and $y$ are the endpoints of $P$, which makes $P$ a path of even length, contradicting the requirements set for the application of the $C_4$-expansion at $uv$ and $xy$.
Thus no tight cut of $G$ can contain both $uv$ and $xy$.
\end{proof}

Two cuts $\Cut{X}$ and $\Cut{Y}$ are called \emph{laminar} if $X \subseteq Y$ or $Y \subseteq X$.
As Lov{\'a}sz \cite{lovasz1987matching} proves, the results of a tight cut decomposition are independent of which tight cuts are chosen and the order in which they are contracted.
Thus it is common to consider a maximal family of pairwise laminar tight cuts which translate into maximal families of pairwise laminar tight cuts for the two tight cut contractions resulting from any member of the original family.
This informs the way in which we state the main theorem of this section.

\begin{theorem}\label{thm:tightcuttracking}
Let $G = (A \cup B, E(G)) \in \Barn$ and let $\FFF$ be an inclusion-wise maximal family of non-trivial, pairwise laminar tight cuts in $G$.
If $H$ is a $C_4$-expansion of $G$ at $uv$ and $xy$, then the following is an inclusion-wise maximal family of non-trivial, pairwise laminar tight cuts in $H$, up to the renaming of at most four edges in $H$:
\[ \FFF' := \{ \CutG{G}{X} \mid \CutG{G}{X} \in \FFF \text{ and } |\{ u,v,x,y \} \cap X| \neq 2 \} . \]
\end{theorem}
\begin{proof}
Let $C \subseteq G$ be the facial cycle containing $uv$ and $xy$ and let $E' = E(H) \setminus E(G)$, which are the edges added as part of the $C_4$-expansion.
We will again use the names provided by \Cref{fig:c4ext} during this proof.
To set up the rest of our arguments, we partition $\FFF$ into three sets:
\begin{itemize}
    \item $\FFF_1 := \{ \CutG{G}{X} \in \FFF \mid |\{ u,v,x,y \} \cap X| \in \{ 0, 4 \} \}$,
    
    \item $\FFF_2 := \{ \CutG{G}{X} \in \FFF \mid |\{ u,v,x,y \} \cap X| \in \{ 1, 3 \} \}$, and
    
    \item $\FFF_3 := \FFF \setminus ( \FFF_1 \cup \FFF_2 ) = \{ \CutG{G}{X} \in \FFF \mid |\{ u,v,x,y \} \cap X| = 2 \}$.
\end{itemize}
We will start by proving that the cuts in $\FFF_1 \cup \FFF_2$ can be translated into tight cuts in $H$.
Let $\CutG{G}{X}$ be a cut in $\FFF_1$ and let $M$ be a perfect matching of $G$, with $ab \in (M \cap \Cut{X})$.
It is easy to observe that $\CutG{G}{X} \cap E' = \emptyset$ in this case.
Therefore $\CutG{G}{X}$ remains a cut in $H$ and we can assume w.l.o.g.\ that $|\{ u,v,x,y \} \cap X| = 0$.
Furthermore, we can assume that $X$ contains all endpoints of the edges of $\CutG{G}{X}$ which lie in $A$, according to \Cref{lem:tightcutcolour}.
From this circumstance it is evident that $\CutG{G}{X}$ is also tight in $H$, as $\InducedSubgraph{G}{X} \subseteq H$ and we must match exactly one vertex from $A \cap X$ with a vertex outside of $X$ to find a perfect matching of $H$.
Thus $\FFF_1$ is a family of non-trivial, laminar tight cuts in $H$.

Next, we consider a cut $\CutG{G}{Y} \in \FFF_2$.
W.l.o.g.\ we assume that $u \in Y$ and $v,x,y \in \overline{Y}$.
To separate $u$ from $v$ in $G$, we must have $uv \in \CutG{G}{Y}$, which also means that $xy \not\in \CutG{G}{Y}$ according to \Cref{lem:c4exptight}.
We claim that $(\CutG{G}{Y} \setminus \{ uv \}) \cup \{ uu' \}$ is a tight cut in $H$.
This can be shown using an argument analogous to the one used in the case from the previous paragraph, as we again have $\InducedSubgraph{G}{Y} \subseteq H$.
This is the situation which requires the renaming of an edge, if we truly wanted to translate the tight cuts of $G$ into $H$.
Note that $|\FFF_2| \leq 4$, since we only have four vertices which can be exclusive to a shore.
This justifies our bound on the number of edges that need to be renamed.

We will now consider $\FFF_3$, or more explicitly, we will prove that all tight cuts in $H$ must be related to the cuts in $\FFF_1 \cup \FFF_2$.
For this discussion, we can assume $\FFF_1 \cup \FFF_2$ to be empty, or equivalently, we can consider $G$ and $H$ to be the results of contracting all tight cuts in $\FFF_1 \cup \FFF_2$ and renaming vertices appropriately, i.e.\ if we contract a shore from a cut of $\FFF_2$ which contained only $u$, but not $v$, $x$, or $y$, then we rename the contraction vertex to $u$.
Since the family of tight cuts in $\FFF$ is laminar, any tight cuts we find in this reduced version of the graphs can also be found if we do not contract $\FFF_1 \cup \FFF_2$, up to some relabelling of edges incident to $u,v,x,y$.
Note that even if we had to contract tight cuts to arrive in this simplified case, we know that according to \Cref{lem:tightcutpreserve} our graphs are still found in $\Barn$.
We claim that under these assumptions $H$ is a brace.

Suppose towards a contradiction that $\CutG{H}{Z}$ is a non-trivial tight cut in $H$ and note that since $H \in \Barn$, according to \Cref{lem:tightcutdisjoint}, the cut $\CutG{H}{Z}$ is a set of three mutually disjoint, pairwise non-adjacent edges.
If $|\{ u,v,x,y \} \cap Z| \in \{ 0, 1, 3, 4 \}$ then $\CutG{H}{Z}$ corresponds to a tight cut in $\FFF_1 \cup \FFF_2$, which we assumed to be empty.
Thus $|\{ u,v,x,y \} \cap Z| = 2$ and therefore some edges from $E'$ must be part of $\CutG{H}{Z}$.

We distinguish three cases, the first being that $u$ and $v$ lie on the same shore of $\CutG{H}{Z}$.
Clearly, $\CutG{H}{Z}$ must involve some edges in $E'$.
However, due to the restrictions \Cref{lem:tightcutcolour} places on the colours of the endpoints of these edges, we would in fact need at least three edges from $E'$ to separate $\{ u, v \}$ and $\{ x, y \}$ in $H - \CutG{H}{Z}$.
Furthermore, due to our limited choices, these edges cannot be mutually disjoint, contradicting \Cref{lem:tightcutdisjoint}. 
Thus we can go on to assuming that $u$ and $v$, and respectively $x$ and $y$, lie on different shores with respect to $\CutG{H}{Z}$.

If $u$ and $x$ lie on the same shore, then $\CutG{H}{Z}$ must again involve at least two edges of $E'$ whilst also respecting the constraints provided by \Cref{lem:tightcutcolour}.
Again, this is not possible if we choose only two such edges from $E'$ and any three suitable edges are not mutually disjoint, again contradicting \Cref{lem:tightcutdisjoint}.

We move on to assuming that $u$ and $y$ lie on the same shore of the tight cut $\CutG{H}{Z}$.
Using the arguments above, especially those based on \Cref{lem:tightcutcolour}, we can deduce that $uu', yy' \in \CutG{H}{Z}$, leaving us with one edge $e \in \CutG{H}{Z} \cap E(G)$.
This implies $\CutG{H}{Z} = \{ uv, xy, e \}$ is tight in $G$, which directly contradicts \Cref{lem:c4exptight}.
\end{proof}

We observe that this theorem also solidifies the intuition that any tight cut in a graph in $\Barn$ was initially introduced by a cube-expansion. 
This idea together with \Cref{thm:tightcuttracking} suggests the following amendments to the generation procedure for $\Barn$.
\begin{itemize}
    \item Associate with each generated graph $G$ a family $\FFF_G$, which contains an inclusion-wise maximal family of non-trivial tight cuts in $G$.
    (As we will see, we do not have to check these properties.)
    
    \item If we use a cube-expansion on a vertex $v$ in a graph $G$ to generate $H$, we add the tight cut created by this operation into $\FFF_G$.
    Any edges of cuts in $\FFF_G$ incident to $v$, are then renamed appropriately.
    The result of this is $\FFF_H$.
    Note that the new tight cut is guaranteed to be laminar to the others, since we expanded a single vertex.
    
    \item Whenever we use a $C_4$-expansion in a graph $G$ at $uv$ and $xy$ to generate $H$, we again rename the edges in cuts of $\FFF_G$ which involve $uv$ or $xy$.
    Furthermore, we remove any cuts $\CutG{G}{X}$ from $\FFF_G$ with $|\{ u,v,x,y \} \cap X| \neq 2$.
    The result of this is $\FFF_H$, which is an inclusion-wise maximal family of non-trivial tight cuts in $G$ according to \Cref{thm:tightcuttracking}.
    Of course if $\FFF_G = \emptyset$, then $\FFF_H = \emptyset$.
    The family $\FFF_H$ can be constructed in linear time (see \Cref{lem:linear}).
\end{itemize}
If we generate the graphs in $\Barn$ in this way, we only have to check the Hamiltonicity of a graph $G \in \Barn$ if $\FFF_G = \emptyset$, which means that it is a brace.
This drives up the space-requirements of the procedure somewhat, but the time saved by not checking graphs with tight cuts for their Hamiltonicity should make up for this fact, at least for small numbers of vertices.
For some pairs of graphs, this also makes it easier to check whether they are isomorphic, since for two isomorphic graphs $G$ and $G'$, we must have $|\FFF_G| = |\FFF_{G'}|$.

In this context it seems pertinent to ask just how many tight cuts we might find and thus have to save in addition to the graph.
We give a sharp upper bound to this.

\begin{lemma}\label{lem:howmanytightcuts}
Let $G \in \Barn$, with $|V(G)| = n$ and let $\FFF$ be an inclusion-wise maximal family of non-trivial, pairwise laminar tight cuts, then $|\FFF| \leq \frac{n-8}{6}$.
This bound is sharp.
\end{lemma}
\begin{proof}
We note that all graphs in $\Barn$ have at least eight vertices and thus $0 \leq \frac{n-8}{6}$, which verifies the bound for all braces in $\Barn$.
Assume that $G$ is a smallest counterexample to the statement of the lemma and observe that $G$ is therefore not a brace.
We note that in any maximal family of non-trivial, pairwise laminar tight cuts, we can always find a cut $\Cut{X}$, such that $X \subseteq Y$, for all $\Cut{Y} \in \FFF$.
Thus $G_1 = \ContractXinGtoV{\overline{X}}{G}{c}$ is a brace, since any tight cut in $G_1$ would translate into a tight cut in $G$ missing from $\FFF$, which was assumed to be maximal.
Furthermore, $G_1 \in \Barn$ by \Cref{lem:tightcutpreserve} and $\FFF' = \FFF \setminus \{ \Cut{X} \}$ is an inclusion-wise maximal family of non-trivial, pairwise laminar tight cuts in $G_2 = \ContractXinGtoV{X}{G}{c}$, up to some renaming of edges.

We note that the smallest brace in $\Barn$ has eight vertices.
Hence, $|X| \geq 7$ and we have $|V(G_2)| \leq n - 6$, since we retained the contraction vertex $c$.
This can then be combined to yield
\[ |\FFF| = |\FFF'| + 1 \leq \frac{n-14}{6} + 1 = \frac{n-8}{6} , \]
which contradicts $G$ being a smallest counterexample, as this entails $|\FFF| > \frac{n-8}{6}$.

An example of a graph with a tight cut which meets this bound can be constructed by applying a single cube-expansion to the cube.
\end{proof}

If we save every tight cut as a set of three edges, we thus need additional space for each generated graph adding up to at most half the number of its vertices.
Let us now argue that we can compute $\FFF_H$ from $\FFF_G$ in linear time if we performed a $C_4$-expansion, under the assumption that $\FFF_G$ was saved in a way that lets us directly link the edges in the graphs with the edges in the cuts contained in $\FFF_G$.
This could for example be realised by labelling the edges of $G$.

\begin{lemma}\label{lem:linear}
Let $G \in \Barn$ and let $\FFF_G$ be an inclusion-wise maximal family of non-trivial, pairwise laminar tight cuts in $G$.
If $H$ is a $C_4$-expansion of $G$ at $uv$ and $xy$, then an inclusion-wise maximal family $\FFF_H$ of non-trivial, pairwise laminar tight cuts in $H$ can be constructed from $\FFF_G$ in linear time.
\end{lemma}
\begin{proof}
We first make some observations on the structure of the graph, which we can then use to simplify our computation.
Let $\CutG{G}{X} \in \FFF_G$ be such that $|\{ u,v,x,y \} \cap X| = 2$ and recall that according to \Cref{thm:tightcuttracking} these are the tight cuts we must identify and remove to construct $\FFF_H$.

We claim that $\CutG{G}{X}$ cannot separate $u$ and $v$.
Suppose that this is not true.
This requires that $uv \in \CutG{G}{X}$.
Now, independent of whether $u$ and $x$, or $u$ and $y$ lie on the same shore, we must also have $xy \in \CutG{G}{X}$, which contradicts \Cref{lem:c4exptight}.
Thus our claim holds and $\CutG{G}{X}$ must separate $u,v$ and $x,y$ in $G$.
In fact, as a consequence of this claim, we know that we must remove a cut from $\FFF_G$ if and only if it separates $u,v$ and $x,y$ in $G$.
Let us call such a cut \emph{removable}.
We further remark that for any removable cut $\CutG{G}{X}$, we therefore have $\CutG{G}{X} \cap \{ uv, xy \} = \emptyset$.

Next, let us consider a path $P \subseteq G$ with the endpoints $u$ and $x$, which must exist since $G$ is connected.
We now analyse how $P$ interacts with some arbitrary $\CutG{G}{Y} \in \FFF_G$.
W.l.o.g.\ we assume that $u \in Y$.
The graph $P - \CutG{G}{Y}$ is composed of components $P_0, \ldots , P_t$ that are paths, which we can name such that $u \in V(P_0)$, $x \in V(P_t)$, and the paths appear in increasing order when we traverse $P$ from $u$ to $x$.
It is now easy to see that $P_i \subseteq \InducedSubgraph{G}{Y}$, for an $i \in \{ 0, \ldots , t \}$, if and only if $i$ is even.
Therefore $\CutG{G}{Y}$ is removable if and only if $t$ is odd, which is equivalent to $|E(P) \cap \CutG{G}{Y}|$ being even.

Let us turn to constructing $\FFF_H$.
We assume that an edge $e \in E(G)$ has the label $\ell_Z$ if $e \in \Cut{Z} \in \FFF_G$.
A path $P$ with endpoints $u$ and $x$ can be found in linear time using a breadth first search.
We can then count the occurrences of all labels, noting the results separately for each label.
If any specific label occurs an odd number of times, we delete this label from all edges of the graph, which can again be done by traversing the graph once after we have determined which labels have to be deleted.
After performing the $C_4$-expansion and the requisite renaming of edges mentioned in \Cref{thm:tightcuttracking}, which can be performed in constant time, the labels which are left form $\FFF_H$.
If necessary $\FFF_H$ can be constructed explicitly by traversing $H$ once more, again adding a linear effort.
\end{proof}

While the above arguments take care of the generation of braces in $\Barn$, the procedure used in \cite{holton1984hamiltonian} to check the Hamiltonicity of the generated graphs involves checking the $H^{+-}$-property for small graphs in $\Barn$.
For larger graphs, we can then find certain configurations that allow us to reduce to the smaller graphs we have generated.
The $H^{+-}$-property can then be used to confirm that the larger graphs are Hamiltonian, though the actual $H^{+-}$-property cannot be preserved.
We illustrate this with the following result.

\begin{theorem}[Brinkmann et al. \cite{brinkmann2021minimality}]\label{thm:computational}
Let $G \in \Barn$, with $|V(G)| = n$, then
\begin{enumerate}
    \item $n \leq 90$ implies that $G$ is Hamiltonian,
    
    \item $n \leq 78$ implies that $G$ is $P_2$-Hamiltonian, and
    
    \item $n \leq 66$ implies that $G$ has the $H^{+-}$-property.
\end{enumerate}
\end{theorem}

The first and second item are implied by the third through the use of several results from \cite{holton1984hamiltonian}.
\Cref{lem:tightcutproperty} tells us that we can preserve the $H^{+-}$-property through tight cuts.
In particular, combining the above theorem and \Cref{lem:tightcutproperty} tells us that, if $G \in \Barn$ is a graph whose braces each have 66 vertices or less, then $G$ has the $H^{+-}$-property.
Thus our amendments to the generation procedure can be adapted in a straight-forward fashion to checking the $H^{+-}$-property.

Finally, we want to mention that using \Cref{lem:tightcutdisjoint}, our results can be rephrased such that any tight cut is replaced with a cut of order three consisting of mutually disjoint, pairwise non-adjacent edges, if we are only concerned with cubic, 3-connected, bipartite graphs.
This basically translates everything into the language of cyclically 4-connected graphs, in line with \Cref{cor:bracecyc4con}.
Whilst this might make some of the statements more intuitive for the reader, we are not sure if a straightforward translation of the proofs exist.

%% file: graphs/claw.tex
			\begin{tikzpicture}[scale=0.75]
			
			\node (V0) at (0:0) [draw, circle, scale=0.6, label={}] {};
			
			\foreach\i in {1,2,3}
			{
				\node (V\i) at ($(V0)+({(360/3 * \i) + 360/4}:2)$) [draw, circle, scale=0.6, fill, label={}] {};
			}
			
			\foreach\i in {1,2,3}
			{
				\path (V\i) edge[very thick] (V0);
			}
			\end{tikzpicture}

%% file: graphs/r4ext.tex
			\begin{tikzpicture}[scale=0.75]
			
			\node (V0) at (0:0) [draw, circle, scale=0.6, label={}] {};
			
			\foreach\i in {1,2,3}
			{
				\node (V\i) at ($(V0)+({(360/3 * \i) + 360/4}:2)$) [draw, circle, scale=0.6, fill, label={}] {};
			}
			\foreach\i in {1,2,3,4,5,6}
			{
				\node (U\i) at ($(V0)+({(360/6 * \i) + 360/4}:1)$) [draw, circle, scale=0.6, label={}] {};
			}
			\node (W1) at ($(V0)+({(360/6 * 1) + 360/4}:1)$) [draw, circle, fill, scale=0.6] {};
			\node (W3) at ($(V0)+({(360/6 * 3) + 360/4}:1)$) [draw, circle, fill, scale=0.6] {};
			\node (W5) at ($(V0)+({(360/6 * 5) + 360/4}:1)$) [draw, circle, fill, scale=0.6] {};
			
			\foreach\i in {1,2,3,4,5}
			{
				\pgfmathtruncatemacro\iplus{\i+1}
				\path (U\i) edge[very thick] (U\iplus);
			}
			\path
				(U6) edge[very thick] (U1)
				(V1) edge[very thick] (U2)
				(V2) edge[very thick] (U4)
				(V3) edge[very thick] (U6)
				(V0) edge[very thick] (U1)
				(V0) edge[very thick] (U3)
				(V0) edge[very thick] (U5)
				;
			\end{tikzpicture}

%% file: graphs/c4expansionpre.tex
			\begin{tikzpicture}[scale=0.75]
			
			\node (V1) at (0,0) [draw, circle, scale=0.5 ,label=below:{$x$}] {};
			\node (V2) at (4,0) [draw, circle, scale=0.5, fill ,label=below:{$y$}] {};
			\node (V5) at (0,2) [draw, circle, scale=0.5, fill ,label={$u$}] {};
			\node (V6) at (4,2) [draw, circle, scale=0.5 ,label={$v$}] {};
			
			\path
				(V1) edge[very thick] (V2)
				(V5) edge[very thick] (V6)
			;
			\end{tikzpicture}

%% file: graphs/c4expansion.tex
			\begin{tikzpicture}[scale=0.75]
			
			\node (V1) at (0,0) [draw, circle, scale=0.5 ,label=below:{$x$}] {};
			\node (V2) at (4,0) [draw, circle, scale=0.5, fill ,label=below:{$y$}] {};
			\node (V3) at (1,0) [draw, circle, scale=0.5, fill ,label=below:{$x'$}] {};
			\node (V4) at (3,0) [draw, circle, scale=0.5 ,label=below:{$y'$}] {};
			\node (V5) at (0,2) [draw, circle, scale=0.5, fill ,label={$u$}] {};
			\node (V6) at (4,2) [draw, circle, scale=0.5 ,label={$v$}] {};
			\node (V7) at (1,2) [draw, circle, scale=0.5 ,label={$u'$}] {};
			\node (V8) at (3,2) [draw, circle, scale=0.5, fill ,label={$v'$}] {};
			
			\path
				(V1) edge[very thick] (V4)
				(V4) edge[very thick] (V2)
				(V5) edge[very thick] (V7)
				(V7) edge[very thick] (V6)
				(V3) edge[very thick] (V7)
				(V4) edge[very thick] (V8)
			;
			\end{tikzpicture}

%% file: graphs/bigbarnetteexample.tex
			\begin{tikzpicture}
			
			\node (A1) at (2.25,0) [draw, circle, scale=0.6] {};
			\node (A2) at (2.25,0.75) [draw, circle, scale=0.6, fill] {};
			\node (A3) at (0,1.5) [draw, circle, scale=0.6, fill] {};
			\node (A4) at (1.5,1.5) [draw, circle, scale=0.6] {};
			\node (A5) at (2.25,1.5) [draw, circle, scale=0.6] {};
			\node (A6) at (2.25,2.25) [draw, circle, scale=0.6, fill] {};
			\node (A7) at (2.25,3) [draw, circle, scale=0.6] {};
			
			\node (B1) at (4,0) [draw, circle, scale=0.6, fill] {};
			\node (B2) at (4,0.75) [draw, circle, scale=0.6] {};
			\node (B3) at (3.625,1.5) [draw, circle, scale=0.6, fill] {};
			\node (B4) at (4.375,1.5) [draw, circle, scale=0.6, fill] {};
			\node (B5) at (4,2.25) [draw, circle, scale=0.6] {};
			\node (B6) at (4,3) [draw, circle, scale=0.6, fill] {};
			
			\node (C1) at (6.5,0) [draw, circle, scale=0.6] {};
			\node (C2) at (6.5,0.75) [draw, blue, circle, scale=0.6, fill] {};
			\node (C3) at (6.125,1.5) [draw, circle, scale=0.6] {};
			\node (C4) at (6.875,1.5) [draw, circle, scale=0.6] {};
			\node (C5) at (6.5,2.25) [draw, circle, scale=0.6, fill] {};
			\node (C6) at (6.5,3) [draw, circle, scale=0.6] {};
			
			\node (D1) at (8.25,0) [draw, circle, scale=0.6, fill] {};
			\node (D2) at (8.25,0.75) [draw, circle, scale=0.6] {};
			\node (D3) at (10.5,1.5) [draw, circle, scale=0.6] {};
			\node (D4) at (9,1.5) [draw, circle, scale=0.6, fill] {};
			\node (D5) at (8.25,1.5) [draw, circle, scale=0.6, fill] {};
			\node (D6) at (8.25,2.25) [draw, circle, scale=0.6] {};
			\node (D7) at (8.25,3) [draw, circle, scale=0.6, fill] {};
			
			\node (E1) at (0.75,4) [draw, circle, scale=0.6] {};
			\node (E2) at (1.5,3.5) [draw, circle, scale=0.6, fill] {};
			\node (E3) at (9,3.5) [draw, circle, scale=0.6] {};
			\node (E4) at (9.75,4) [draw, circle, scale=0.6, fill] {};

            \path
			(E1) edge[very thick, red] (A3)
			(E1) edge[very thick, red] (E2)
			(E1) edge[very thick] (E4)
			(E2) edge[very thick, red] (A7)
			(E2) edge[very thick] (E3)
			(E3) edge[very thick, red] (E4)
			(E3) edge[very thick, red] (D7)
			(E4) edge[very thick, red] (D3)
			;
			
			\path
			(A1) edge[very thick] (A2)
			(A1) edge[very thick] (A3)
			(A2) edge[very thick] (A4)
			(A2) edge[very thick] (A5)
			(A3) edge[very thick] (A4)
			(A4) edge[very thick] (A6)
			(A5) edge[very thick] (A6)
			(A6) edge[very thick] (A7)
			;
			
			\path
			(B1) edge[very thick] (B2)
			(B2) edge[very thick] (B3)
			(B2) edge[very thick] (B4)
			(B3) edge[very thick] (B5)
			(B4) edge[very thick] (B5)
			(B5) edge[very thick] (B6)
			;
			
			\path
			(C1) edge[very thick] (C2)
			(C2) edge[very thick] (C3)
			(C2) edge[very thick] (C4)
			(C3) edge[very thick] (C5)
			(C4) edge[very thick] (C5)
			(C5) edge[very thick] (C6)
			;
			
			\path
			(D1) edge[very thick] (D2)
			(D1) edge[very thick] (D3)
			(D2) edge[very thick] (D4)
			(D2) edge[very thick] (D5)
			(D3) edge[very thick] (D4)
			(D4) edge[very thick] (D6)
			(D5) edge[very thick] (D6)
			(D6) edge[very thick] (D7)
			;
			
			\path
			(A1) edge[very thick, dashed] (B1)
			(A5) edge[very thick, dashed] (B3)
			(A7) edge[very thick, dashed] (B6)
			(B1) edge[very thick, dashed] (C1)
			(B4) edge[very thick, dashed] (C3)
			(B6) edge[very thick, dashed] (C6)
			(C1) edge[very thick, dashed] (D1)
			(C4) edge[very thick, dashed] (D5)
			(C6) edge[very thick, dashed] (D7)
			;
			\end{tikzpicture}

%% file: furtherresearch.tex
We mentioned splices briefly in \Cref{sec:note} and said that reversed versions of \Cref{thm:tightcutham} let us construct new Hamiltonian graphs.
This can in fact be done with \Cref{lem:tightcutpreserve} and, though the lemma does not concern tight cuts, \Cref{lem:bracestosum} can be used similarly.
In particular, the $K_{3,3}$ and the Heawood graph are $P_3$- and $P_4$-Hamiltonian, and have the $H^-$- and the $H^{+-}$-property.
Additionally, \Cref{thm:computational} provides us with a large class of small cubic braces with the $H^-$- and the $H^{+-}$-property.
Thus we can generate a large class containing both planar and non-planar, cubic, bipartite graphs with strong Hamiltonicity properties by repeatedly using splices.

In \Cref{sec:pfaffian}, we both mention a reduction of \Cref{con:barnette} to braces and then provide what is essentially a relaxation of planarity for the conjecture.
With this reformulated conjecture, we can again attempt to relax or strengthen certain properties.
Let us start with planar, cubic braces, which are 3-connected by \Cref{thm:exttocon}.
Using Euler's formula, it is fairly easy to prove that a planar, bipartite graph cannot have minimum degree four.
Thus we cannot demand 4-connectivity.
We cannot drop planarity either, as the Georges-Kelmans graph $GK$, discovered by Georges \cite{georges1989non} and Kelmans \cite{kelmans1986konstruktsii}\cite{kelmans1988cubic} independently, is a brace, as the authors verified by hand.
Thus the only option left is to drop regularity, leading us to the following question.

\begin{question}
    Is every planar brace Hamiltonian?
\end{question}

Concerning non-Pfaffian graphs, it is fairly easy to find a conformal bisubdivision of the $K_{3,3}$ in $GK$, but we did not check whether all counterexamples to Tutte's conjecture are Pfaffian.
If any of them are, then this would disprove \Cref{con:barnette} via the sixth item of \Cref{thm:mainthmpfaffian}.
Whilst investigating non-planar braces, the following question comes naturally in the vein of the main subject of \cite{brinkmann2021minimality} and $GK$ already gives an upper bound of $50$.

\begin{question}
    What is the least number of vertices of a non-planar, non-Hamiltonian brace?
\end{question}

One of the key lemmas for proving \Cref{thm:mainthmpfaffian} was \Cref{lem:tightcutproperty}.
Missing from the items there are both $P_2$-Hamiltonicity and Hamiltonicity itself.
If we have a $P_2$-Hamiltonian, cubic, $3$-connected, bipartite graph, then it is easy to show, using the methods we presented, that all of its tight cut contractions will also be $P_2$-Hamiltonian and for Hamiltonicity this fact is stated in \Cref{thm:tightcutham}.
One might wonder if the other direction holds for either property.

This however can be disproven by considering the first graph presented by Horton \cite{bondy1976graph} as a counterexample to Tutte's conjecture.
First we note that this graph can be constructed by taking a $K_{3,3}$ and splicing a copy of the graph in \Cref{fig:hortonfrag} into each of the three vertices of one of its colour classes.
Each of these splices creates a tight cut.
The fact that the Horton graph can be constructed this way confirms that it is not Pfaffian according to \Cref{thm:tightcutpfaffian} and \Cref{thm:pfaffianlittle}, since the $K_{3,3}$ is one of its braces.
We have already noted that the $K_{3,3}$ has several strong Hamiltonicity properties.
Unsurprisingly, it is also $P_2$-Hamiltonian.
This is also true for the graph in \Cref{fig:hortonfrag}, which also happens to be non-Pfaffian.
Thus if two tight cut contractions of a given cubic, $3$-connected, bipartite graph $G$ are ($P_2$-)Hamiltonian, the graph $G$ itself may still be non-Hamiltonian.

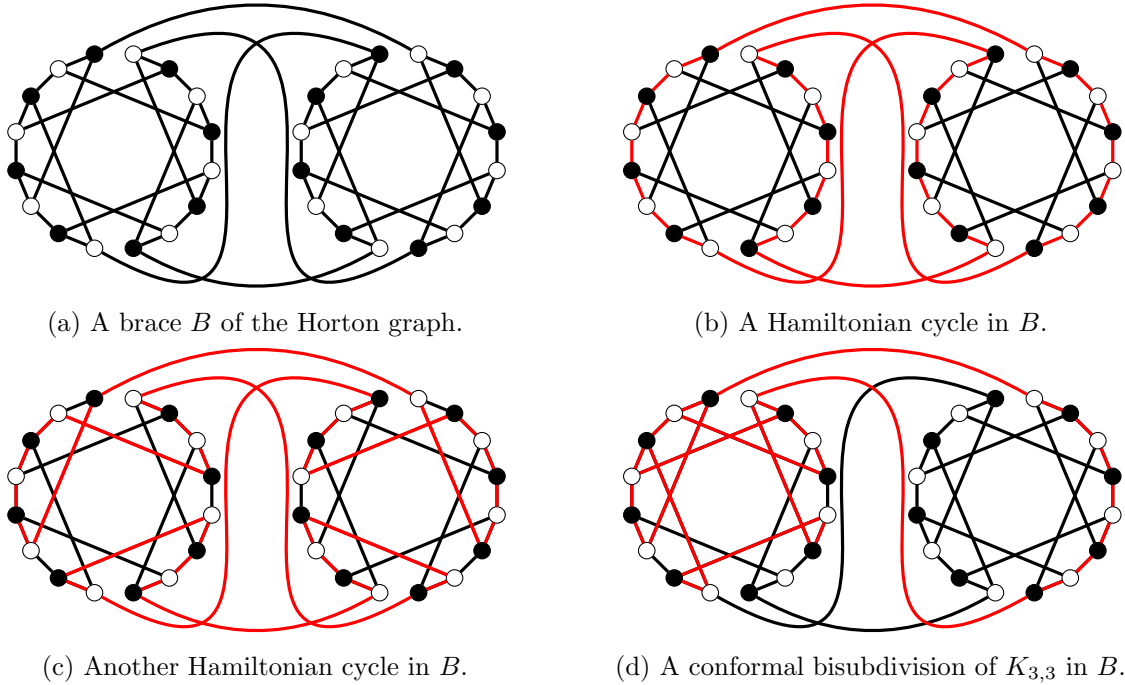
\begin{figure}[ht]
     \centering
     \begin{subfigure}[b]{0.45\textwidth}
         \centering
         \input{graphs/hortonfragment.tex}
         \caption{A brace $B$ of the Horton graph.}
         \label{fig:hortonfrag}
     \end{subfigure}
     \qquad
     \begin{subfigure}[b]{0.45\textwidth}
         \centering
         \input{graphs/hortonham1.tex}
         \caption{A Hamiltonian cycle in $B$.}
         \label{fig:hortonham1}
     \end{subfigure} \\
     \begin{subfigure}[b]{0.45\textwidth}
         \centering
         \input{graphs/hortonham2.tex}
         \caption{Another Hamiltonian cycle in $B$.}
         \label{fig:hortonham2}
     \end{subfigure}
     \qquad
     \begin{subfigure}[b]{0.45\textwidth}
         \centering
         \input{graphs/hortonk33.tex}
         \caption{A conformal bisubdivision of $K_{3,3}$ in $B$.}
         \label{fig:hortonk33}
     \end{subfigure}
     \caption{The two Hamiltonian cycles given for $B$ should help to confirm that $B$ is $P_2$-Hamiltonian.
     Due to the symmetry in $B$, the cycle in \Cref{fig:hortonham2} can be shifted such that all edges not covered by the cycle in \Cref{fig:hortonham1} can be found in a Hamiltonian cycle of this type in $B$.
     The bisubdivision of $K_{3,3}$ in \Cref{fig:hortonk33} confirms that $B$ is not Pfaffian.}
\end{figure}

In contrast to this, we are unsure whether the first point of \Cref{lem:tightcutproperty} can be reversed.
This is somewhat surprising, since for the other Hamiltonicity properties we considered, this direction was actually the easier one to prove.

\begin{question}
    Let $G$ be a $P_4$-Hamiltonian, $3$-connected, cubic, bipartite graph, with a non-trivial tight cut $\Cut{X}$, and let $G_1$ and $G_2$ be the two tight cut contractions belonging to $\Cut{X}$.
    Are $G_1$ and $G_2$ also $P_4$-Hamiltonian?
\end{question}

Regarding our suggestions for the generation procedure for $\Barn$, one might ask why we cannot just check each generated graph and see whether it is a brace.
This is indeed possible.
A polynomial time algorithm for determining the extendability of a bipartite graph was first given by Lakhal and Litzler \cite{lakhal1998polynomial} and an improvement to a runtime in $\mathcal{O}(mn)$, which is in $\mathcal{O}(n^2)$ for cubic graphs, was then provided by Zhang and Zhang \cite{zhang2006construction}.
Of course, we could also just check whether the generated graph is cyclically 4-connected, but the known algorithms (see \cite{dvovrak2004algorithm} and \cite{lu2009efficient}) do not improve on a quadratic runtime.
Either way, unless space is scarce and we have much more processor time to burn, it seems like checking whether a generated graph is a brace in this way would not be advisable, though it might still improve on simply generating all graphs and checking Hamiltonicity indiscriminately.

\textbf{Acknowledgements.} We want to thank Tom Fowler for providing us with an unpublished manuscript \cite{fowler2000reducible}, in which he tackles the task of finding more reductions for Barnette's Conjecture.
This ended up informing the discussion on known reductions in \Cref{sec:generation}.

Additionally, we want to thank an anonymous reviewer for spotting an error in the original bound of \Cref{lem:howmanytightcuts}, pointing us to the Horton graph as an example of a non-Hamiltonian graph with ($P_2$-)Hamiltonian braces, and providing several smaller remarks which improved the paper.

%% file: graphs/hortonfragment.tex
    		\begin{tikzpicture}[scale=0.75]
    		
    		\clip (-2,-2.5) rectangle (7,2.7);
			
			\node (V0) at (0:0) [draw=none] {};
			
			\foreach\i in {1,3,5,7,9,11,13,15}
			{
				\node (V\i) at ($(V0)+({(360/16 * \i)+(360/4)-(360/32)}:1.75)$) [draw, circle, scale=0.6, fill, label={}] {};
			}
			
			\foreach\i in {2,4,6,8,10,12,14,16}
			{
				\node (V\i) at ($(V0)+({(360/16 * \i)+(360/4)-(360/32)}:1.75)$) [draw, circle, scale=0.6, label={}] {};
			}
			
			\foreach\i in {1,2,3,4,5,6,7,9,10,11,12,13,14,15}
			{
				\pgfmathtruncatemacro\iplus{\i+1}
				\path (V\i) edge[very thick] (V\iplus);
			}
			
			\path
				(V1) edge[very thick] (V6)
				(V2) edge[very thick] (V13)
				(V3) edge[very thick] (V8)
				(V4) edge[very thick] (V15)
				(V5) edge[very thick] (V10)
				(V7) edge[very thick] (V12)
				(V9) edge[very thick] (V14)
				(V11) edge[very thick] (V16)
			;
			
			\node (U0) at (0:5) [draw=none] {};
			
			\foreach\i in {1,3,5,7,9,11,13,15}
			{
				\node (U\i) at ($(U0)+({(360/16 * \i)+(360/4)-(360/32)}:1.75)$) [draw, circle, scale=0.6, fill, label={}] {};
			}
			
			\foreach\i in {2,4,6,8,10,12,14,16}
			{
				\node (U\i) at ($(U0)+({(360/16 * \i)+(360/4)-(360/32)}:1.75)$) [draw, circle, scale=0.6, label={}] {};
			}
			
			\foreach\i in {1,2,3,4,5,6,7,9,10,11,12,13,14,15}
			{
				\pgfmathtruncatemacro\iplus{\i+1}
				\path (U\i) edge[very thick] (U\iplus);
			}
			
			\path
				(U1) edge[very thick] (U6)
				(U2) edge[very thick] (U13)
				(U3) edge[very thick] (U8)
				(U4) edge[very thick] (U15)
				(U5) edge[very thick] (U10)
				(U7) edge[very thick] (U12)
				(U9) edge[very thick] (U14)
				(U11) edge[very thick] (U16)
			;
			
			\path
				(V9) edge[very thick, bend right] (U8)
				(V1) edge[very thick, bend left] (U16)
			;
			
			\node (C1) at ($(U0)+({(360/16 * 3)+(360/4)-(360/32)}:6.75)$) [draw=none, label={}] {};
			\node (C2) at ($(U0)+({(360/16 * 6)+(360/4)}:6.25)$) [draw=none, label={}] {};
			\node (C3) at ($(V0)+({(360/16 * 11)+(360/4)-(360/16)}:6.25)$) [draw=none, label={}] {};
			\node (C4) at ($(V0)+({(360/16 * 14)+(360/4)-(360/32)}:6.75)$) [draw=none, label={}] {};
			
			\draw[very thick] (V8) .. controls (C3) and (C1) .. (U1);
			\draw[very thick] (V16) .. controls (C4) and (C2) .. (U9);
			
			\end{tikzpicture}

%% file: graphs/hortonham1.tex
    		\begin{tikzpicture}[scale=0.75]
    		
    		\clip (-2,-2.5) rectangle (7,2.7);
			
			\node (V0) at (0:0) [draw=none] {};
			
			\foreach\i in {1,3,5,7,9,11,13,15}
			{
				\node (V\i) at ($(V0)+({(360/16 * \i)+(360/4)-(360/32)}:1.75)$) [draw, circle, scale=0.6, fill, label={}] {};
			}
			
			\foreach\i in {2,4,6,8,10,12,14,16}
			{
				\node (V\i) at ($(V0)+({(360/16 * \i)+(360/4)-(360/32)}:1.75)$) [draw, circle, scale=0.6, label={}] {};
			}
			
			\foreach\i in {1,2,3,4,5,6,7,9,10,11,12,13,14,15}
			{
				\pgfmathtruncatemacro\iplus{\i+1}
				\path (V\i) edge[red, very thick] (V\iplus);
			}
			
			\path
				(V1) edge[very thick] (V6)
				(V2) edge[very thick] (V13)
				(V3) edge[very thick] (V8)
				(V4) edge[very thick] (V15)
				(V5) edge[very thick] (V10)
				(V7) edge[very thick] (V12)
				(V9) edge[very thick] (V14)
				(V11) edge[very thick] (V16)
			;
			
			\node (U0) at (0:5) [draw=none] {};
			
			\foreach\i in {1,3,5,7,9,11,13,15}
			{
				\node (U\i) at ($(U0)+({(360/16 * \i)+(360/4)-(360/32)}:1.75)$) [draw, circle, scale=0.6, fill, label={}] {};
			}
			
			\foreach\i in {2,4,6,8,10,12,14,16}
			{
				\node (U\i) at ($(U0)+({(360/16 * \i)+(360/4)-(360/32)}:1.75)$) [draw, circle, scale=0.6, label={}] {};
			}
			
			\foreach\i in {1,2,3,4,5,6,7,9,10,11,12,13,14,15}
			{
				\pgfmathtruncatemacro\iplus{\i+1}
				\path (U\i) edge[red, very thick] (U\iplus);
			}
			
			\path
				(U1) edge[very thick] (U6)
				(U2) edge[very thick] (U13)
				(U3) edge[very thick] (U8)
				(U4) edge[very thick] (U15)
				(U5) edge[very thick] (U10)
				(U7) edge[very thick] (U12)
				(U9) edge[very thick] (U14)
				(U11) edge[very thick] (U16)
			;
			
			\path
				(V9) edge[red, very thick, bend right] (U8)
				(V1) edge[red, very thick, bend left] (U16)
			;
			
			\node (C1) at ($(U0)+({(360/16 * 3)+(360/4)-(360/32)}:6.75)$) [draw=none, label={}] {};
			\node (C2) at ($(U0)+({(360/16 * 6)+(360/4)}:6.25)$) [draw=none, label={}] {};
			\node (C3) at ($(V0)+({(360/16 * 11)+(360/4)-(360/16)}:6.25)$) [draw=none, label={}] {};
			\node (C4) at ($(V0)+({(360/16 * 14)+(360/4)-(360/32)}:6.75)$) [draw=none, label={}] {};
			
			\draw[red, very thick] (V8) .. controls (C3) and (C1) .. (U1);
			\draw[red, very thick] (V16) .. controls (C4) and (C2) .. (U9);
			
			\end{tikzpicture}

%% file: graphs/hortonham2.tex
    		\begin{tikzpicture}[scale=0.75]
    		
    		\clip (-2,-2.5) rectangle (7,2.7);
			
			\node (V0) at (0:0) [draw=none] {};
			
			\foreach\i in {1,3,5,7,9,11,13,15}
			{
				\node (V\i) at ($(V0)+({(360/16 * \i)+(360/4)-(360/32)}:1.75)$) [draw, circle, scale=0.6, fill, label={}] {};
			}
			
			\foreach\i in {2,4,6,8,10,12,14,16}
			{
				\node (V\i) at ($(V0)+({(360/16 * \i)+(360/4)-(360/32)}:1.75)$) [draw, circle, scale=0.6, label={}] {};
			}
			
			\foreach\i in {1,2,3,4,5,6,7,9,10,11,12,13,14,15}
			{
				\pgfmathtruncatemacro\iplus{\i+1}
				\path (V\i) edge[very thick] (V\iplus);
			}
			\foreach\i in {2,3,4,5,7,9,10,11,13,14,15}
			{
				\pgfmathtruncatemacro\iplus{\i+1}
				\path (V\i) edge[red, very thick] (V\iplus);
			}
			
			\path
				(V1) edge[red, very thick] (V6)
				(V2) edge[red, very thick] (V13)
				(V3) edge[very thick] (V8)
				(V4) edge[very thick] (V15)
				(V5) edge[very thick] (V10)
				(V7) edge[red, very thick] (V12)
				(V9) edge[very thick] (V14)
				(V11) edge[very thick] (V16)
			;
			\path
				(V1) edge[red, very thick] (V6)
				(V2) edge[red, very thick] (V13)
				(V7) edge[red, very thick] (V12)
			;
			
			\node (U0) at (0:5) [draw=none] {};
			
			\foreach\i in {1,3,5,7,9,11,13,15}
			{
				\node (U\i) at ($(U0)+({(360/16 * \i)+(360/4)-(360/32)}:1.75)$) [draw, circle, scale=0.6, fill, label={}] {};
			}
			
			\foreach\i in {2,4,6,8,10,12,14,16}
			{
				\node (U\i) at ($(U0)+({(360/16 * \i)+(360/4)-(360/32)}:1.75)$) [draw, circle, scale=0.6, label={}] {};
			}
			
			\foreach\i in {1,2,3,4,5,6,7,9,10,11,12,13,14,15}
			{
				\pgfmathtruncatemacro\iplus{\i+1}
				\path (U\i) edge[very thick] (U\iplus);
			}
			\foreach\i in {1,2,3,5,6,7,9,11,12,13,14}
			{
				\pgfmathtruncatemacro\iplus{\i+1}
				\path (U\i) edge[red, very thick] (U\iplus);
			}
			
			\path
				(U1) edge[very thick] (U6)
				(U2) edge[very thick] (U13)
				(U3) edge[very thick] (U8)
				(U4) edge[red, very thick] (U15)
				(U5) edge[red, very thick] (U10)
				(U7) edge[very thick] (U12)
				(U9) edge[very thick] (U14)
				(U11) edge[red, very thick] (U16)
			;
			\path
				(U4) edge[red, very thick] (U15)
				(U5) edge[red, very thick] (U10)
				(U11) edge[red, very thick] (U16)
			;
			
			\path
				(V9) edge[red, very thick, bend right] (U8)
				(V1) edge[red, very thick, bend left] (U16)
			;
			
			\node (C1) at ($(U0)+({(360/16 * 3)+(360/4)-(360/32)}:6.75)$) [draw=none, label={}] {};
			\node (C2) at ($(U0)+({(360/16 * 6)+(360/4)}:6.25)$) [draw=none, label={}] {};
			\node (C3) at ($(V0)+({(360/16 * 11)+(360/4)-(360/16)}:6.25)$) [draw=none, label={}] {};
			\node (C4) at ($(V0)+({(360/16 * 14)+(360/4)-(360/32)}:6.75)$) [draw=none, label={}] {};
			
			\draw[red, very thick] (V8) .. controls (C3) and (C1) .. (U1);
			\draw[red, very thick] (V16) .. controls (C4) and (C2) .. (U9);
			
			\end{tikzpicture}

%% file: graphs/hortonk33.tex
    		\begin{tikzpicture}[scale=0.75]
    		
    		\clip (-2,-2.5) rectangle (7,2.7);
			
			\node (V0) at (0:0) [draw=none] {};
			
			\foreach\i in {1,3,5,7,9,11,13,15}
			{
				\node (V\i) at ($(V0)+({(360/16 * \i)+(360/4)-(360/32)}:1.75)$) [draw, circle, scale=0.6, fill, label={}] {};
			}
			
			\foreach\i in {2,4,6,8,10,12,14,16}
			{
				\node (V\i) at ($(V0)+({(360/16 * \i)+(360/4)-(360/32)}:1.75)$) [draw, circle, scale=0.6, label={}] {};
			}
			
			\foreach\i in {1,2,3,4,5,6,7,9,10,11,12,13,14,15}
			{
				\pgfmathtruncatemacro\iplus{\i+1}
				\path (V\i) edge[very thick] (V\iplus);
			}
			\foreach\i in {1,2,3,4,5,7,11,13,14,15}
			{
				\pgfmathtruncatemacro\iplus{\i+1}
				\path (V\i) edge[red, very thick] (V\iplus);
			}
			
			\path
				(V1) edge[very thick] (V6)
				(V2) edge[very thick] (V13)
				(V3) edge[very thick] (V8)
				(V4) edge[very thick] (V15)
				(V5) edge[very thick] (V10)
				(V7) edge[very thick] (V12)
				(V9) edge[very thick] (V14)
				(V11) edge[very thick] (V16)
			;
			\path
				(V1) edge[red, very thick] (V6)
				(V2) edge[red, very thick] (V13)
				(V3) edge[red, very thick] (V8)
				(V4) edge[red, very thick] (V15)
				(V7) edge[red, very thick] (V12)
				(V11) edge[red, very thick] (V16)
			;
			
			\node (U0) at (0:5) [draw=none] {};
			
			\foreach\i in {1,3,5,7,9,11,13,15}
			{
				\node (U\i) at ($(U0)+({(360/16 * \i)+(360/4)-(360/32)}:1.75)$) [draw, circle, scale=0.6, fill, label={}] {};
			}
			
			\foreach\i in {2,4,6,8,10,12,14,16}
			{
				\node (U\i) at ($(U0)+({(360/16 * \i)+(360/4)-(360/32)}:1.75)$) [draw, circle, scale=0.6, label={}] {};
			}
			
			\foreach\i in {1,2,3,4,5,6,7,9,10,11,12,13,14,15}
			{
				\pgfmathtruncatemacro\iplus{\i+1}
				\path (U\i) edge[very thick] (U\iplus);
			}
			\foreach\i in {9,10,11,12,13,14,15}
			{
				\pgfmathtruncatemacro\iplus{\i+1}
				\path (U\i) edge[red, very thick] (U\iplus);
			}
			
			\path
				(U1) edge[very thick] (U6)
				(U2) edge[very thick] (U13)
				(U3) edge[very thick] (U8)
				(U4) edge[very thick] (U15)
				(U5) edge[very thick] (U10)
				(U7) edge[very thick] (U12)
				(U9) edge[very thick] (U14)
				(U11) edge[very thick] (U16)
			;
			
			\path
				(V9) edge[very thick, bend right] (U8)
				(V1) edge[red, very thick, bend left] (U16)
			;
			
			\node (C1) at ($(U0)+({(360/16 * 3)+(360/4)-(360/32)}:6.75)$) [draw=none, label={}] {};
			\node (C2) at ($(U0)+({(360/16 * 6)+(360/4)}:6.25)$) [draw=none, label={}] {};
			\node (C3) at ($(V0)+({(360/16 * 11)+(360/4)-(360/16)}:6.25)$) [draw=none, label={}] {};
			\node (C4) at ($(V0)+({(360/16 * 14)+(360/4)-(360/32)}:6.75)$) [draw=none, label={}] {};
			
			\draw[very thick] (V8) .. controls (C3) and (C1) .. (U1);
			\draw[red, very thick] (V16) .. controls (C4) and (C2) .. (U9);
			
			\end{tikzpicture}